\documentclass[11pt,a4paper]{amsart}
\usepackage[dvipsnames]{xcolor}
\usepackage[a-3u]{pdfx}
\usepackage[utf8]{inputenc}
\usepackage[T1]{fontenc}
\usepackage{amsmath,amsthm,amssymb}
\usepackage{mathtools}
\usepackage{thmtools}
\usepackage{thm-restate}
\usepackage{graphicx}
\usepackage{tikz, tikz-cd}
\usepackage[all]{xy}
\usepackage[inline]{enumitem}
\usepackage{bm}

\usepackage{hyperref}
\hypersetup{unicode,breaklinks=true,colorlinks=true,allcolors=blue}
\usepackage{cleveref}

\numberwithin{equation}{section}
\newtheorem{thm}{}[section]
\newtheorem{theorem}[thm]{Theorem}

\newtheorem{lemma}[thm]{Lemma}
\newtheorem{prop}[thm]{Proposition}
\theoremstyle{definition}
\newtheorem{example}[thm]{Example}
\newtheorem{defn}[thm]{Definition}
\newtheorem{question}[thm]{Question}
\newtheorem{rem}[thm]{Remark}

\newtheorem{thminside}{}

\newtheorem{claim}[thminside]{Claim}
\newcommand{\abs}[1]{\left\lvert#1\right\rvert}
\newcommand{\norm}[1]{\left\lVert#1\right\rVert}

\newcommand{\enbrace}[1]{\left\lbrace#1\right\rbrace}
\newcommand{\enbrak}[1]{\left[#1\right]}
\newcommand{\enpar}[1]{\left(#1\right)}
\newcommand{\lc}{\ensuremath{\bm{cv}}}
\newcommand{\mz}{\ensuremath{\bm{d}}}
\newcommand{\lmz}{\ensuremath{\bm{c}}}
\newcommand{\umz}{\ensuremath{\bm{u}}}
\newcommand{\smz}{\ensuremath{\bm{sd}}}
\newcommand{\slmz}{\ensuremath{\bm{sc}}}
\newcommand{\sumz}{\ensuremath{\bm{su}}}
\newcommand{\hsd}{\ensuremath{\bm{\chi}}}
\newcommand{\krt}{\ensuremath{\bm{\alpha}}}
\newcommand{\itc}{\ensuremath{\bm{\beta}}}

\newcommand{\FF}{\ensuremath{\mathbb{F}}}
\newcommand{\VV}{\ensuremath{\mathbb{V}}}
\newcommand{\XX}{\ensuremath{\mathbb{X}}}
\newcommand{\YY}{\ensuremath{\mathbb{Y}}}
\newcommand{\EB}{\ensuremath{\mathcal{E}}}
\newcommand{\XB}{\ensuremath{\mathcal{X}}}
\newcommand{\YB}{\ensuremath{\mathcal{Y}}}
\newcommand{\Mt}{\mathcal{M}}
\newcommand{\Nt}{\mathcal{N}}
\newcommand{\Ft}{\mathcal{F}}
\newcommand{\Gt}{\mathcal{G}}
\newcommand{\It}{\mathcal{I}}

\newcommand{\ee}{\ensuremath{\bm{e}}}
\newcommand{\Id}{\ensuremath{\mathrm{Id}}}
\DeclareMathOperator{\supp}{supp}

\DeclareMathOperator{\Lip}{Lip}

\DeclareMathOperator{\diam}{diam}
\DeclareMathOperator{\Span}{span}
\DeclareMathOperator{\suppt}{suppt}
\newcommand{\Nat}{\mathbb{N}}
\newcommand{\ZZ}{\mathbb{Z}}
\newcommand{\Rea}{\mathbb{R}}
\newcommand{\F}{\mathcal{F}}

\author[F. Albiac]{Fernando Albiac}
\address{Department of Mathematics, Statistics and Computer Sciences, and Ina\-Mat$^{2}$\\ Universidad P\'ublica de Navarra\\
Pamplona 31006\\ Spain}
\email{fernando.albiac@unavarra.es}
\author[J. L. Ansorena]{Jos\'e L. Ansorena}
\address{Department of Mathematics and Computer Sciences\\
Universidad de La Rioja\\
Logro\~no 26004\\ Spain}
\email{joseluis.ansorena@unirioja.es}
\author[J. B\'{\i}ma]{Jan B\'{\i}ma}
\address{Faculty of Mathematics and Physics, Department of Mathematical Analysis\\
Charles University\\
186 75 Praha 8\\
Czech Republic}
\email{jan.bima@mff.cuni.cz}
\author[M. C\'uth]{Marek C\'uth}
\address{Faculty of Mathematics and Physics, Department of Mathematical Analysis\\
Charles University\\
186 75 Praha 8\\
Czech Republic}
\email{cuth@karlin.mff.cuni.cz}
\begin{document}
%--------------------------------------------
\title[On the Schur $p$-Property and Compact Reduction in  $\F_p(\Mt)$]{Lipschitz free $\bm p$-spaces for $\bm{0<p<1}$\\ in the light of the Schur $\bm{p}$-property\\ and the compact reduction}
%--------------------------------------------
\begin{abstract}
The geometric analysis of non-locally convex quasi-Banach spaces presents rich and nuanced challenges. In this paper, we introduce the Schur $p$-property and the strong Schur $p$-property for $0 < p \leq 1$, providing new tools to deepen the understanding of these spaces, and the Lipschitz free $p$-spaces in particular. Moreover, by developing an adapted version of the compact reduction principle, we prove that Lipschitz free $p$-spaces over discrete metric spaces possess the approximation property, thereby answering positively a question raised by Albiac et al. in \cite{AACD21}.
\end{abstract}
%--------------------------------------------
\thanks{F. Albiac and J. L. Ansorena acknowledge the support of the Spanish Ministry for Science and Innovation under Grant PID2022-138342NB-I00 for \emph{Functional Analysis Techniques in Approximation Theory and Applications}. J. Bíma and M. Cúth acknowledge the support of GAČR project 23-04776S. J. Bíma further acknowledges the support of the Charles University project GA UK No. 138123 and completed this work while employed at MSD Czech Republic, Prague.}
%--------------------------------------------
\maketitle
%--------------------------------------------
\section{Introduction}\noindent
%--------------------------------------------
Lipschitz free $p$-spaces over metric or quasi-metric spaces $\Mt$, here denoted $\F_p(\Mt)$, were introduced in \cite{AlbiacKalton2009} with the aim of constructing, for each $0 < p < 1$, examples of two \emph{separable} $p$-Banach spaces that are Lipschitz-isomorphic but not linearly isomorphic. Whether this is possible or not for $p=1$ remains, as of today, the single most important open problem in the theory of non-linear classification of Banach spaces.

Lipschitz free spaces over metric spaces have attracted the attention of specialists in nonlinear geometric analysis and classical Banach space theory alike since the publication of the book \cite{WeaverBook2018} by Weaver, and the seminal paper \cite{GodefroyKalton2003} by Godefroy and Kalton. This upsurge of interest has contributed to a very fast development of the theory and a better understanding of this class of spaces. Despite the fact that many simple-to-formulate questions remain unanswered, this area is currently a very active focus of research in functional and metric analysis as well as their interactions.

In contrast, the study of Lipschitz free $p$-spaces for $0<p<1$ has lagged far behind compared to the case when $p = 1$, despite the fact that the spaces $\F_{p}(\Mt)$ have proven to be of substantial utility. The historic neglect in the advancement of the non-locally convex theory in any front has its roots way back at the dawn of functional analysis, and it is justified by the difficulty in making headway when some of the main tools the analyst relies on in their everyday life (duality, differentiation, integration) are not available when the local convexity is lifted. We contend here that the shortage of techniques does not make $p$-Banach spaces for $0<p<1$ less interesting, but rather more intricate. This intricacy makes their idiosyncrasies well worth exploring if we wish to gain a deeper appreciation of the theory as a whole.

With that aim in mind, in this paper we bring into the non-locally convex setting two techniques that have been recently and successfully applied in the context of Lipschitz free spaces to investigate their structure, namely the Schur property and the compact reduction.

In functional analysis, the Schur property serves to identify those Banach spaces $\XX$ where the weak and the norm topologies behave similarly on sequences. The first infinite-dimensional Banach space where this property was uncovered was $\ell_1$, as Schur himself showed in 1921. Banach spaces with the Schur property are weakly sequentially complete, but the converse does not hold in general. Indeed, $L_1$ and the reflexive spaces are examples of weakly sequentially complete spaces which do not enjoy the Schur property.

As a feature that permits to tell apart spaces which are not isomorphic, the Schur property has many applications in classical Banach space theory. For instance, $L_{1}$ cannot be isomorphic to $\ell_{1}$ because the former contains the non-Schur space $\ell_2$ complementably, and the Schur property is inherited by subspaces. Or one could simply appeal to the fact that $\ell_1$ and $L_1$ have different ways to interact with their respective duals as far as the convergence of sequences is concerned.

The interplay between a Banach space and its dual is lost in the absence of local convexity, as it happens in the spaces $L_p$ for $0<p<1$ or in any other quasi-Banach spaces with zero dual for that matter. Even when the dual space of a quasi-Banach space is huge (like in $\ell_p$ for $0<p<1$ whose dual is $\ell_\infty$), the information one could get by studying the connections between the norm and the week topologies is gone. Still, it makes sense to see if it would be possible to rescue some vestiges of the Schur property through some equivalent characterizations in the context of quasi-Banach spaces with an eye to its applications to the isomorphic theory. To that end, in Section~\ref{sec:schur} we introduce the Schur $p$-property of a quasi-Banach space $\XX$. Roughly speaking, we will be interested in identifying $\ell_p$-like behaviour inside a quasi-Banach space and, more specifically, in knowing when we can extract a sequence which is equivalent to the canonical $\ell_{p}$-basis from any uniformly separated sequence in the space. When $p=1$ this is a reformulation of the standard Schur property for Banach spaces, courtesy of Rosenthal's $\ell_1$ theorem.

Schur's property is a rare find in ``classical'' Banach space theory. At the same time, since $\Rea$ embeds in any Banach space $\XX$, the Lipschitz free space $\F(\XX)$ contains $L_1$ isomorphically, so it cannot be a Schur space. However, Lipschitz free spaces over metric spaces provide many examples of Banach spaces that enjoy this property. In \cite{Kalton2004}, Kalton proved that if $(\Mt, d)$ is a metric space and $\omega$ is a nontrivial gauge then the Lipschitz free space $\F(\Mt_\omega)$ over the metric space $M_{\omega}=(M, \omega\circ d)$ has the Schur property. In the same paper, he also proved that if $\Mt$ is uniformly discrete and uniformly separated then $\F(\Mt)$ is a Schur space with the Radon-Nikodym property and the approximation property. The relevance of the Schur property in the study of the geometry of Lipschitz free spaces over metric spaces has gained relevance in the last few years and has been exploited by several authors. For instance, in \cite{HLP2016}  H\'ajek et al.\@ proved that the Lipschitz free space over a proper countable metric space has the Schur property. In turn, Petitjean investigated the Schur property in the class of Lipschitz free spaces over metric spaces originating from $p$-Banach spaces when $0<p<1$. To be precise, if $(\XX, \norm{\cdot})$ is a $p$-Banach space with a finite-dimensional decomposition and $\Mt_{p}$ denotes the metric space $(\XX, d_{p})$, where $d_{p}(x,y)=\norm{ x-y}^{p}$ ($x$, $y\in \XX)$, then the Lipschitz free space $\F(\Mt_p)$ over $\Mt_p$ has the Schur property for all $0<p<1$ \cite{Petitjean2017}; in particular, $\F(\ell_p, \norm{\cdot}_p^p)$ has the Schur property for all $0<p<1$. A final characterization was found in \cite{AGPP22} by Aliaga et al.\@ who proved that $\F(\Mt)$ has the Schur property if and only if the completion of $\Mt$ is purely $1$-unrectifiable. Thus the time was ripe to study this feature in Lipschitz free $p$-spaces over quasi-metric spaces.

The Schur $p$-property is paired with and compared to the also new strong Schur $p$-property, which strengthens it quantitatively. Once again, this demonstrates a certain consistency in the theory with the locally convex case when $p = 1$. The stability of the Schur property by infinite $\ell_p$-sums is established in Section~\ref{stability}. This is relevant also seen from the angle that embedding a space into an $\ell_p$-direct sum is another way to exhibit $\ell_{p}$-like behaviour. We move to Section~\ref{free} with some early applications to the understanding of the geometry of the Lipschitz free $p$-spaces $\F_{p}(\Mt)$ over metric or quasi-metric spaces $\Mt$ for $0<p\le 1$ through the lens of the Schur $p$-property. In Section~\ref{cpctReduction} we incorporate (to some extent) the, so called, compact reduction method to the setting of Lipschitz free $p$-spaces for $p<1$. Compact reduction in the context of Lipschitz free spaces refers to a technique used to study the structure of these spaces by relating them to compact subsets of the underlying metric spaces. It was introduced and successfully applied by Aliaga et al. in \cite{ANPP} to show, for instance, that if $\Mt$ is a scattered complete metric space then $\F(\Mt)$ has the Schur property.  The full force of this technique is attained when combined with duality, which imposes certain limitations for its usefulness to non-locally convex $p$-Banach spaces.  Here we develop a weak version of a compact reduction principle and show that if $\Mt$ is a discrete metric space then $\F_p(\Mt)$ has the approximation property, thus answering in the positive \cite[Question 6.3]{AACD21} (the case of $p=1$ had been stablished in \cite{ANPP}).
%--------------------------------------------------------------
\section{The Schur \texorpdfstring{$p$}{}-property and the strong Schur \texorpdfstring{$p$}{}-property}\label{sec:schur}\noindent
%--------------------------------------------------------------
A Banach space $\XX$ (over the real or complex field $\FF$) has the \emph{Schur property} if whenever a sequence in $\XX$ converges weakly to $0$ then is norm-convergent to $0$. If we consider this property in the more general class of quasi-Banach spaces, we obtain that  any quasi-Banach space with the Schur property is locally convex. So, in order for the Schur property to provide some meaningful information on the structure of a quasi-Banach space $\XX$, the space must be locally convex.

To adapt the Schur property to the geometry of non-locally convex spaces we appeal to a ready consequence Rosenthal's $\ell_1$-theorem, which tells us that every bounded sequence in a Banach space either has a weakly Cauchy subsequence or contains a subsequence equivalent to the canonical $\ell_1$-basis.

The following proposition is well-known but we include it here for self-reference. Recall that a family $\XB=(x_j)_{j\in J}$ in a metric (or quasi-metric) space $(\Mt,d)$ is said to be \emph{uniformly separated} if
\[
\delta(\XB)\coloneqq\inf_{j\not=k} d(x_j,x_k)>0.
\]
If $\delta\in(0,\infty)$ is such that $\delta(\XB)\ge\delta$, we say that $\XB$ is $\delta$-separated. A set $A\subset\XX$ is said to be $\delta$-separated if the obvious family $(x)_{x\in A}$ is. A set (or a family) in a quasi-Banach space is said to be \emph{semi-normalized} if the quasi-norms of its vectors are bounded away from zero and infinity. If $A$ is contained in the unit sphere $S_\XX$ of $\XX$, we say that it is \emph{normalized}.

\begin{prop}\label{prop:SchurChar}
A Banach space $\XX$ has the Schur property if and only if every uniformly separated bounded infinite subset of $\XX$ contains a sequence equivalent to the canonical $\ell_1$-basis.
\end{prop}

\begin{proof}
If $\XX$ fails to have the Schur property, then it contains is a semi-normalized weakly null sequence $\XB\coloneqq(x_n)_{n=1}^\infty$. By the Hahn--Banach theorem,
\[
\liminf_n \norm{x-x_n}\ge \norm{x}
\]
for all $x\in\XX$. Hence, by Cantor's diagonal method, $\XB$ has a uniformly separated subsequence $\YB$. Since the canonical $\ell_1$-basis is not weakly null, no subsequence of $\YB$ is equivalent to the canonical $\ell_1$-basis.

Suppose that there is a uniformly separated bounded sequence $\XB$ in $\XX$ with no subbasis equivalent to the canonical $\ell_1$-basis. Then, $\XB$ has a weakly Cauchy subsequence $(x_n)_{n=1}^\infty$. Since the sequence
\[
(x_{2n-1}-x_{2n})_{n=1}^\infty
\]
is semi-normalized and weakly null, we are done.
\end{proof}

We shall define a Schur $p$-property by replacing $\ell_1$ with $\ell_p$ in the characterization of the Schur property provided by \Cref{prop:SchurChar}.

\begin{defn}\label{def:pSchur}
We will say that a quasi-Banach space $\XX$ has the \emph{Schur $p$-property}, or that or $\XX$ is a \emph{$p$-Schur space}, $0<p\le 1$, if every bounded uniformly separated infinite subset of $\XX$ contains a sequence equivalent to the canonical basis of $\ell_p$.
\end{defn}

The Schur property is of qualitative nature. In contrast, the Schur $p$-property, $0<p\le 1$, involves notions that depend on the existence of inherent numerical magnitudes. Below we will define a quantitative Schur $p$-property by imposing suitable relations between those implicit constants. Let us first introduce some notation we will need.

The symbol $[I]^{<\omega}$ stands for the set of all finite subsets of a set $I$. Given a quasi-Banach space $\XX$ and a nonempty set $B\subset \XX$, we set
\[
\lmz_p(B)\coloneqq\inf\enbrace{\norm{\sum_{x\in F} a_x \, x}\colon F\in[B]^{<\omega}, \, (a_x)_{x\in F}\in S_{\ell_p(F)}},
\]
that is, $\lmz_p(B)$ is the optimal constant $c\in[0,\infty)$ such that
\[
c \enpar{\sum_{x\in F} \abs{a_x}^p}^{1/p} \le \norm{\sum_{x\in F} a_x \, x}
\]
for every eventually null family $(a_x)_{x\in B}$ in $\FF$. Similarly, with the convention that $1/\infty = 0$, we put
\[
\umz_p(B)\coloneqq\enpar{\sup\enbrace{\norm{\sum_{x\in F} a_x \, x}\colon F\in[B]^{<\omega}, \, a\in S_{\ell_p(F)}}}^{-1}.
\]
That is, $\umz_p(B)$ is the optimal $c\in [0,\infty)$ satisfying
\[
c \norm{\sum_{x\in F} a_x \, x} \le \enpar{\sum_{x\in F} \abs{a_x}^p}^{1/p}
\]
for every eventually null family $(a_x)_{x\in B}$ in $\FF$.

Then, we measure the distance from $B$ to the canonical $\ell_p$-basis by means of the quantity
\[
\mz_p(B)\coloneqq\min\enbrace{\lmz_p(B),\umz_p(B)}.
\]
Since
\begin{equation}\label{eq:TB}
\lmz_p(B)\le \inf \enbrace{ \norm{x} \colon x\in B}, \quad \umz_p(B)\le \frac{1}{\sup\enbrace{\norm{x}\colon x\in B}},
\end{equation}
it follows that $\mz_p(B)\le 1$.

Let $\enbrak{B}$ denote the closed linear span of $B$. Let $(\ee_x)_{x\in B}$ denote the unit vectors of $\FF^B$. Notice that $\lmz_p(B)>0$ (resp., $\umz_p(B)>0$) if and only if there is a bounded linear operator $T\colon \enbrak{B} \to \ell_p(B)$ (resp., $S\colon \ell_p(B)\to \enbrak{B}$) such that $T(x)=\ee_x$ (resp., $S(\ee_x)=x$) for all $x\in B$, in which case $\lmz_p(B) = \norm{T}^{-1}$ (resp., $\umz_p(B) =\norm{S}^{-1}$).

If $A\subset \XX$ is infinite, we put
\begin{align*}
\smz_p(A) &\coloneqq \sup\enbrace{ \mz_p(B) \colon B\subset A, \, \abs{B}=\infty},\\
\slmz_p(A) &\coloneqq \sup\enbrace{ \lmz_p(B) \colon B\subset A, \, \abs{B}=\infty},\\
\sumz_p(A) &\coloneqq \sup\enbrace{ \umz_p(B) \colon B\subset A, \, \abs{B}=\infty}.\\
\end{align*}

We define the \emph{diametral modulus} of a quasi-metric space $(\Mt,d)$ as
\[
\rho(\Mt)\coloneqq \sup \enbrace{\tfrac{d(x,z)}{\max\enbrace{d(x,y),d(y,z) }}\colon x,y,z\in\Mt, \, x\neq y\neq z}.
\]
Note that $\rho(\Mt)\in[1,\infty)$ is the optimal constant $C$ such that
\[
d(x,z)\le C \max\enbrace{d(x,y),d(y,z) }, \quad x\,,y\,,z\in \Mt.
\]
The diametral modulus $\rho(\XX)$ of a quasi-normed space $(\XX,\norm{\cdot})$ will be the diametral modulus of the associated quasi-metric on $\XX$ given by $(x,y)\mapsto \norm{x-y}$. In this case $\rho(\XX)\in[2,\infty)$ and it measures the diameter of the unit ball $B_\XX$ of $\XX$. It is also the optimal constant $C$ such that
\[
\norm{x+y} \le C\max\enbrace{\norm{x},\norm{y}}, \quad x,y \in\XX.
\]

Let $\XX$ be a quasi-Banach space. Let us record some elementary properties of the canonical projection
\[
\pi\colon\XX\setminus\{0\}\to S_\XX, \quad \pi(x)\coloneqq\frac{x}{\norm{x}},
\]
whose straightforward verification we omit.
\begin{lemma}\label{lem:Cuth} Let $\XX$ be a quasi-Banach space with diametral modulus $\rho$.
\begin{enumerate}[label=(\roman*),leftmargin=*,widest=iii]
\item\label{it:NormIneq} For all $x$, $y\in \XX\setminus\{0\}$ we have
\begin{equation*}
\norm{x-y} \le \rho \max\enbrace{\min\enbrace{ \norm{x}, \norm{y}}\norm{\pi(x)-\pi(y)}, \abs{ \norm{y}-\norm{x}}}.
\end{equation*}
In particular, if $A\subset \XX\setminus\{0\}$ is such that,
\[
\norm{x-y}>\rho \abs{ \norm{y}-\norm{x}}, \quad x,y\in A, \, x\not=y,
\]
then $\pi|_A$ is one-to-one.

\item\label{it:MZNormalized} Given $B\subset \XX\setminus\{0\}$ such that $\pi|_B$ is one-to-one,
\begin{align*}
\lmz_p(B)&\ge \lmz_p(\pi(B)) \inf\enbrace{\norm{x}\colon x\in B}, \\
\umz_p(B)&\ge \umz_p(\pi(B)) \frac{1}{\sup\enbrace{\norm{x}\colon x\in B}}.
\end{align*}
\end{enumerate}
\end{lemma}

\begin{lemma}
Let $\XX$ be a quasi-Banach space and $p\in(0,1]$. The following are equivalent.
\begin{itemize}[leftmargin=*]
\item $\XX$ has the Schur $p$-property.
\item $\smz_p(A)>0$ for every $A\subset\XX$ infinite, bounded, and uniformly separated.
\item $\smz_p(A)>0$ for every $A\subset S_\XX$ infinite and uniformly separated.
\item $\slmz_p(A)>0$ and $\sumz_p(A)>0$ for every $A\subset\XX$ infinite, bounded, and uniformly separated.
\item $\slmz_p(A)>0$ and $\sumz_p(A)>0$ for every $A\subset S_\XX$ infinite and uniformly separated.
\end{itemize}
\end{lemma}
\begin{proof}

Given $A\subset \XX$ infinite, bounded and uniformly separated, and $\varepsilon>0$, there are a closed interval $I\subset (0,\infty)$ of length $\varepsilon$ and $A_0\subset A$ infinite such that $\norm{x}\in I$ for all $x\in A_0$. Choosing $\varepsilon$ small enough, we infer from Lemma~\ref{lem:Cuth}\ref{it:NormIneq} that $\pi|_{A_0}$ is one-to-one and that $\pi(A_0)$ is uniformly separated. Since subsets of $\delta$-separated sets are $\delta$-separated, the equivalences follow from Lemma~\ref{lem:Cuth}\ref{it:MZNormalized}.
\end{proof}

We now introduce the strong Schur $p$-property for $0 < p \leq 1$, which, as we shall see, is strictly stronger than the Schur $p$-property. For the origin of the strong Schur property in the context of Banach spaces, we refer the reader to \cite[Section 2]{Ha81}.

\begin{defn}\label{df:SSchur}
Let $\XX$ be a quasi-Banach space, $p\in (0,1]$, and $K>0$. We say that $\XX$ has \emph{$K$-strong Schur $p$-property} if for every $\delta>0$, every $K_0>K$, and every infinite $\delta$-separated set $A\subset S_\XX$, there exists an infinite subset $B \subset A$ such that
\begin{equation*}
\lmz_p(B)\geq \frac{\delta}{K_0 2^{1/p}}, \quad \umz_p(B)\ge \frac{1}{K_0}.
\end{equation*}
We will say that $\XX$ has the strong Schur $p$-property if it has the $K$-strong Schur $p$-property for some $K\in(0,\infty)$.
\end{defn}

By Lemma~\ref{lem:Cuth}, in \Cref{df:SSchur} we can replace the unit sphere with an annulus by paying the price of increasing the constant $K$. So both the Schur $p$-property and the strong Schur $p$-property are preserved by isomorphisms. It is also rather obvious that subspaces of a quasi-Banach space $\XX$ inherit the Schur $p$-property and the strong Schur $p$-property from $\XX$.

It is known that a quasi-Banach space contains an infinite uniformly separated bounded set if and only if it is infinite-dimensional. Quantitatively, we have the following.

\begin{prop}[cf.\@ {\cite[Lemma 2.8]{AlbiacAnsorena2021b}}]\label{prop:QBUS}
Given an infinite-dimensional quasi-Banach space $\XX$ and $0<\delta<1$, there is $A\subset S_\XX$ infinite and $\delta$-separated.
\end{prop}

\begin{proof}
As in the proof of \cite[Lemma 2.8]{AlbiacAnsorena2021b}, we construct $(y_n)_{n=1}^\infty$ in $S_\XX$ inductively. Let $k\geq 1$ and suppose that $y_1, \ldots, y_{k-1}$ have already been chosen. Set $\YY_{k-1} = [y_n  \colon 1\le n \le k-1]$. Choose $x\in \XX$ such that $s_k \coloneqq \norm{-x+\YY_{k-1}} \in (\delta, 1)$. Then, there exists $x_k \in \XX$ such that $x + x_k \in \YY_{k-1}$ and $t_k \coloneqq \norm{x_k} < 1$. Define $y_k = x_k / t_k$. For each $1 \le n \le k-1$, \[ \norm{y_k-y_n}=\frac{1}{t_k} \norm{ x_k + x- t_k y_n- x} \ge \frac{s_k}{t_k} \ge\delta\text. \]

This completes the inductive construction, and the claim follows.
\qedhere
\end{proof}

This yields, on one hand, that any finite-dimensional space trivially has the Schur $p$-property for any $p$ and, on the other hand, that there exists an infinite-dimensional quasi-Banach space with the Schur $p$-property if and only if $\ell_p$ is a Schur $p$-space. So, the consistency of the theory depends on the fact that, as we will see, $\ell_p$ itself is a $p$-Schur space. In this respect, we point out that, although \Cref{def:pSchur} makes a priori sense for $p$ in the whole range $(0,\infty]$, it is of no interest when $p>1$. This can be inferred from the fact that the unit vector system $\EB_p\coloneqq (\ee_n)_{n=1}^\infty$ of $\ell_p$ ($c_0$ if $p=\infty$) is weakly null. Indeed, if $f\in\ell_p\setminus\{0\}$, the sequence $\XB\coloneqq (f+\ee_n)_{n=1}^\infty$ is uniformly separated and bounded and does not have a weakly null subsequence. Hence, no subsequence of $\XB$ is equivalent to $\EB_p$.

Recall that, by the Aoki-Rolewicz theorem \cite{Aoki1942,Rolewicz1957}, every quasi-Banach space $\XX$ is \emph{locally $q$-convex} for some $0<q\le 1$. That is, there exists a constant $\gamma \geq 1$ such that
\begin{equation}\label{eq:lqc}
\norm{\sum_{j\in J} x_j} \leq \gamma \left( \sum_{j\in J} \norm{x_j}^q \right)^{1/q}
\end{equation}
for every finite family $(x_j)_{j\in J}$ in $\XX$. In fact, we can choose $q=\log^{-1}_2(\rho)$, where $\rho$ is the diametral modulus of $\XX$. Moreover, any locally $q$-convex quasi-Banach space $\XX$ is a $q$-Banach under renorming, i.e., it can be equipped with an equivalent $q$-norm. Given a quasi-Banach space $\XX$, let $\lc(\XX)\in(0,1]$ be the supremum of all $q\in(0,1]$ such that $\XX$ is locally $q$-convex.

We point out that there are quasi-Banach spaces $\XX$ that are not locally $\lc(\XX)$-convex. For instance, the weak Lorentz space $L_{p,\infty}$, $0<p\le 1$, is locally $q$-convex for all $0<q<p$, and fails to be locally $p$-convex. We also point out that there are quasi-norms that fail to be continuous maps relative to the topology they induce (see \cite{Hyers1939}).

In any case, $q$-norms are uniformly continuous maps. Although we will not, a priori, require the quasi-norms with which the spaces are equipped to be $q$-norms for some $q\in(0,1]$, we will, at certain points, assume that these quasi-norms are uniformly continuous. Specifically, this means that if we define
\[
\varepsilon(t)\coloneqq\sup\enbrace{ \abs{\norm{x+y}-\norm{x}} \colon x,y\in \XX, \, \norm{y}\le t}, \quad t>0,
\]
then $\lim_{t\to 0^+} \varepsilon(t)=0$. The map $\varepsilon$ is called the \emph{modulus of continuity} of $\norm{\cdot}$.

\begin{rem}
Since $\lc(\ell_p)=p$, we have $\lc(\XX)\le p$ for any infinite-dimensional quasi-Banach $p$-Schur space $\XX$. As we will show, there are instances of infinite-dimensional quasi-Banach spaces $\XX$ with the Schur $p$-property and $\lc(\XX)<p$. In any case, since $\ell_q$ is not a subspace of $\ell_p$ unless $p=q$ \cite{Stiles1970}, there is at most one index $p\in(0,1]$ so that $\XX$ has the Schur $p$-property.
\end{rem}

\begin{rem}
By \eqref{eq:TB}, an infinite-dimensional quasi-Banach space $\XX$ cannot have the $K$-strong Schur $p$-property for $K<1$.
Furthermore, if
\begin{equation*}
\slmz_p(A) \geq \frac{\delta}{K 2^{1/p} }
\end{equation*}
for every infinite $\delta$-separated set $A\subset S_\XX$, then $K\ge 1$. Indeed, for any $a>1$ and $\varepsilon>0$, there are $\delta>1/a$ and $A\subset S_\XX$ infinite such that
\[
\delta<\norm{x-y}<\delta+\varepsilon, \quad x,y\in A, \, x\not=y\text,
\]
see \Cref{prop:QBUS} and \cite[Proposition 23]{Koszmider25}. The set $A$ is $\delta$-separated, and
\[
\slmz_p(A)\leq \frac{\delta+\varepsilon}{2^{1/p}} = \frac{\delta+\varepsilon}{\delta} \frac{\delta}{2^{1/p}} < \enpar{1+a\varepsilon }\frac{\delta}{2^{1/p}}.
\]
Hence, $K\ge 1/(1+a \varepsilon)$. Fixing $a$ and letting $\varepsilon$ tend to zero, we get $K\ge 1$.
\end{rem}

Note that a quasi-Banach space $\XX$ has the \emph{$K$-strong Schur $p$-property} if and only if
\begin{equation*}
\min\enbrace{\frac{2^{1/p} \slmz_p(A)}{\delta}, \sumz_p(A) } \geq \frac{1}{K}
\end{equation*}
for every $\delta>0$ and every $\delta$-separated set in $S_\XX$. We go on by observing that we can replace normalized sets with bounded ones in this inequality.
\begin{lemma}\label{ekvivStrongPSchur}
Let $\XX$ be a quasi-Banach space equipped with a uniformly continuous quasi-norm. Assume that $\XX$ has the $K$-strong Schur $p$-property for some $0<p\le 1$ and $1\le K<\infty$. 

Let $A\subset \XX$ be infinite, bounded, and $\delta$-separated for some $\delta>0$. Put $M\coloneqq\sup\enbrace{\norm{x}\colon x\in A}$. Then
\[
\slmz_p(A)\geq \mu\coloneqq\frac{\delta}{K 2^{1/p}}, \quad
\sumz_p(A)\geq \nu\coloneqq\frac{1}{M K}.
\]
\end{lemma}

\begin{proof}
Leaving the zero vector from $A$ if necessary, we can assume that $m\coloneqq\inf\enbrace{\norm{x}\colon x\in A}>0$. Pick
\[
0<\varepsilon_0<\varepsilon<\min\enbrace{\frac{\delta}{M}, m}.
\]
Let $\eta>0$ be such that $\norm{x-y}\le \eta$ implies $ \abs{ \norm{x}-\norm{y} }\le\varepsilon_0$. There is an interval $I$ of length $\eta m$ such that $A_0\coloneqq\enbrace{ x\in A \colon \norm{x}\in I}$ is infinite. For $x$, $y\in A_0$ we have
\[
\norm{ \frac{x}{\norm{x}} - \frac{x}{\norm{y}}} = \frac{\abs{ \norm{x}-\norm{y} }}{\norm{y}}\le \eta,
\]
whence
\[
\abs{\norm{ \frac{y}{\norm{y}} - \frac{x}{\norm{x}}} - \norm{ \frac{y}{\norm{y}} - \frac{x}{\norm{y}}}}\le \varepsilon_0.
\]
Set $m_0\coloneqq\inf\enbrace{\norm{x}\colon x\in A_0}$ and $M_0\coloneqq\sup\enbrace{\norm{x}\colon x\in A_0}$. We infer that the set $\pi(A_0)$ is $\enpar{\delta/M_0-\varepsilon_0}$-separated and $\pi|_{A_0}$ is one-to-one. Consequently, there is $B\subset A_0$ infinite such that
\[
\lmz_p(\pi(B)) \ge \frac{1}{M_0} \frac{\delta-\varepsilon M}{K 2^{1/p}}, \quad \umz_p(\pi(B)) \ge \frac{m-\varepsilon }{m K}.
\]
Since $m\le M_0 \le m_0+\varepsilon$, applying Lemma~\ref{lem:Cuth}\ref{it:MZNormalized} we obtain
\[
\lmz_p(B)\ge \mu(\varepsilon)\coloneqq \frac{m-\varepsilon}{m}\frac{\delta-\varepsilon M}{K 2^{1/p}},
\quad
\umz_p(B)\ge \nu(\varepsilon)\coloneqq \frac{1}{M}\frac{m-\varepsilon}{m K}.
\]

Letting $\varepsilon$ tend to zero, we get $\mu(\varepsilon)$ arbitrarily close to $\mu$, and $\nu(\varepsilon)$ arbitrarily close to $\nu$.
\end{proof}

Some quantitative measures of noncompactness are useful for estimating $\slmz_p(A)$. Based on \cite{KKS13}, we gather three of them for $p$-metric spaces. We will use the convention that the supremum of an empty set is zero while its infimum is infinity.

Let $(\Mt,d)$ be a quasi-metric space. Given $A$, $B\subset \Mt$, we let
\[
\widehat{d}(A,B)\coloneqq\sup_{a\in A}d(a,B).
\]
Given $A\subset \Mt$, we define
\begin{align*}
\hsd(A)&\coloneqq\inf\enbrace{\widehat{d}(A,F)\colon F\in[\Mt]^{<\omega}\setminus\{\emptyset\}},\\
\hsd_0(A)&\coloneqq\inf\enbrace{\widehat{d}(A,F)\colon F\in[A]^{<\omega}\setminus\{\emptyset\}},\\
\itc(A)&\coloneqq\sup\enbrace{\delta>0\colon \mbox{there exists $B\subset A$ infinite and $\delta$-separated}},\\
\krt(A)&\coloneqq\inf\{\varepsilon>0\colon \text{$A$ is finitely covered by sets of diameter less than $\varepsilon$}\}.
\end{align*}

Note that these quantities are finite if and only if $A$ is bounded. For metric spaces, $\hsd(A)$ is the classical Hausdorff measure of noncompactness, $\krt(A)$ the classical Kuratowski measure of noncompactness, $\hsd_0(A)$ is known as the inner Hausdorff measure of noncompactness, and $\itc(A)$ is the Istratescu measure of noncompactness.

Let $\rho$ be the diametral modulus of $d$. It is easy to check that, since $\diam\enpar{B(x,r)}\leq \rho r$ for all $x\in \Mt$ and $r\in(0,\infty)$,
\begin{equation*}
\hsd(A)\leq \hsd_0(A)\leq \itc(A) \leq \krt(A)\leq \rho \hsd(A), \quad A\subset \Mt.
\end{equation*}
Moreover, $A$ is compact if and only if $\krt(A)=0$.

We present several results that link these measures of noncompactness to the ability to extract subsequences equivalent to the $\ell_p$-basis.

\begin{prop}\label{fact:scAndChi}
Let $0<p\le 1$, $\XX$ be a quasi-Banach space with diametral modulus $\rho$, and $A\subset \XX$. Then
\[
\slmz_p(A)\leq \rho 2^{-1/p} \hsd(A).
\]
In particular, if $\XX$ is a $p$-Banach space, then $\slmz_p(A)\leq \hsd(A)$.
\end{prop}
\begin{proof}
If $0<K<\rho^{-1} 2^{1/p} \slmz_p(A)$ there is an infinite set $B\subset A$ with $\lmz_p(B)\ge \rho 2^{-1/p} K$. Therefore,
\[
\norm{x-y} =2^{1/p} \frac{\norm{x-y}}{2^{1/p}} \geq 2^{1/p} \lmz_p(B) \ge \rho K, \quad x,\, y\in B,\, x\neq y.
\]
Pick $F\subset A$ finite and assume by contradiction that $\widehat{d}(A,F)< K$. Then, for each $x\in A$ there is $z(x)\in F$ such that $\norm{x-z(x)}<K$. Pick $x$, $y\in B$, $x\not=y$, with $z(x)=z(y)$. We deduce that $\norm{x-y}<\rho K$. This absurdity puts an end to the proof.
\end{proof}

Although the Schur $p$-property, $0<p\le 1$, can be studied in quasi-Banach spaces which are not locally $p$-convex, to state our next result we restrict ourselves to $p$-Banach spaces. Note that if $\XX$ is a $p$-Banach space and $B\subset \XX$ is nonempty, then
\[
\umz_p(B)= \frac{1}{\sup\enbrace{\norm{x}\colon x\in B}}.
\]
Hence, given $A\subset\XX$ infinite,
\begin{equation}\label{eq:smzpC}
\sumz_p(A)=\frac{1}{\limsup_{x\in A} \norm{x}}.
\end{equation}

\begin{prop}%\label{strongEquiv}
Let $0<p\le 1$ and $\XX$ be a $p$-Banach space. Given $K\geq 1$, the following conditions are equivalent.
\begin{enumerate}[label=(\roman*),leftmargin=*,widest=iii]
\item\label{it:strongSchur} $\XX$ has the $K$-strong Schur $p$-property.
\item\label{it:inequalityWeak} $\itc(A)\le K 2^{1/p} \slmz_p(A)$ for every $A\subset S_\XX$.
\item\label{it:inequalityStronger} $\itc(A)\le K 2^{1/p} \slmz_p(A)$ for every bounded set $A\subset \XX$.
\end{enumerate}
\end{prop}

\begin{proof}
Clearly, $\XX$ has the $K$-strong Schur $p$-property if and only if $\itc(A) \le K 2^{1/p} \slmz_p(A)$ and $1\le K \sumz_p(A)$ for every $A\subset S_\XX$ infinite. Hence \ref{it:strongSchur} and \ref{it:inequalityWeak} are equivalent by \eqref{eq:smzpC}. Since \ref{it:inequalityStronger} is formally stronger than \ref{it:inequalityWeak}, we have to prove that \ref{it:strongSchur} implies \ref{it:inequalityStronger}. To that end, we pick a bounded set $A\subset \XX$ with $\itc(A)>0$. Given $0<\delta<\itc(A)$, there exists an infinite $\delta$-separated set $B\subset A$. By \Cref{ekvivStrongPSchur},
\[
\slmz_p(A)\geq \slmz_p(B)\geq \frac{\delta}{K 2^{1/p}}.
\]
Taking the limit as $\delta \to \itc(A)$ yields the desired inequality.
\end{proof}

We finish this section with an auxiliary results that will be useful in extracting subsequences equivalent to the canonical basis of $\ell_p$.

\begin{lemma}\label{lem:dichothomy}
Let $\XB$ be a bounded sequence in a complete quasi-metric space $(\Mt,d)$. Then $\XB$ contains either a uniformly separated subsequence or a convergent subsequence.
\end{lemma}

\begin{proof}
By infinite Ramsey theory, for every $\varepsilon>0$ and every subsequence $\YB$ of $\XB$ there are a further subsequence $(z_n)_{n=1}^\infty$ of $\YB$ and $\delta=\delta(\varepsilon,\YB)\in[0,\infty)$ such that $d(z_n,z_k)\in[\delta, \delta+\varepsilon]$ or all $n$, $k\in\Nat$ with $n\not=k$. If $\delta(\varepsilon,\YB)>0$ for some $\varepsilon>0$ and some subsequence $\YB$ of $\XB$, we are done. Otherwise, we pick a null sequence $(\varepsilon_n)_{n=1}^\infty$. By diagonal Cantor's technique, there is a subsequence $(x_n)_{n=1}^\infty$ of $\XB$ such that $d(x_n,x_k)\le \varepsilon_n$ for all $n$, $k\in\Nat$ with $k>n$. Since $(x_n)_{n=1}^\infty$ is a Cauchy sequence, we are done.
\end{proof}
%-----------------------------------------------------
\section{The Schur property of \texorpdfstring{$\ell_p$}{}-sums}\label{stability}\noindent
%------------------------------------------------------
To state and prove the results in this section, we first introduce some terminology. Given $0<p\le \infty$ and a family $(\XX_i)_{i\in \It}$ of quasi-Banach spaces with diametral modulus uniformly bounded, consider the space
\[
\XX\coloneqq\enpar{\bigoplus_{i\in \It} \XX_i}_{\ell_p}=\enbrace{x=(x_i)_{i\in \It} \in\prod_{i\in \It} \XX_i, \, \sigma(x)\coloneqq\enpar{\norm{x_i}}_{i\in \It}\in\ell_p}
\]
equipped with the quasi-norm $\norm{\sigma(\cdot)}_p$. Then $\XX$ is a quasi-Banach space. Moreover, if $0 < p \leq 1$ and each $\XX_i$ is a $p$-Banach space, then $\XX$ is a $p$-Banach space.

Given $x\in \XX$ and $i\in \It$, $x(i)$ denotes the $i$th coordinate of $x$. The support of $x$ is the set
\[
\supp(x)\coloneqq\enbrace{ i\in \It \colon x(i)\not=0}.
\]
Given $\Ft\subset \It$, $x|_\Ft$ stands for the vector in $\XX_\Ft\coloneqq (\bigoplus_{i\in \Ft} \XX_i)_{\ell_p}$ obtained by restricting $x$ to $\Ft$. We denote by $L_\Ft\colon \XX_\Ft\to \XX$ the canonical lifting.

Given $A\subset\XX$, $i\in\It$, and $\Ft\subset \It$, we set
\[
A(i)\coloneqq\enbrace{x(i) \colon x\in A}, \quad A|_\Ft\coloneqq\enbrace{ x|_\Ft \colon x\in A}.
\]

In order to check the Schur property and the strong Schur property, we can restrict ourselves to countable sets. Hence, since the support of any vector in $\XX$ is countable, to check the Schur property and the strong Schur property of $\XX$ it suffices to check it on $\XX_\Ft$ for all $\Ft\subset \It$ countable.

Our first result states that $\ell_p$-sums of finite-dimensional spaces are $p$-Schur spaces. In particular, $\ell_p$ is a $p$-Schur space for each $0<p\le 1$.

In what follows, we will use the inequality
\begin{equation}\label{eq:1lessp}
\norm{\alpha}_1\le \norm{\alpha}_p, \quad \alpha\in\FF^{\Nat}, \, 0<p\le 1.
\end{equation}

\begin{theorem}\label{thm:lp}
Let $(\XX_i)_{i\in I}$ be a family of finite-dimensional quasi-Banach spaces with uniformly bounded diametral modulus and let $0<p\le 1$. Then $\XX\coloneqq (\bigoplus_{i\in \It} \XX_i)_{\ell_p}$ has the strong Schur $p$-property. Moreover, if $\XX_i$ is a $p$-Banach space for all $i\in \It$, then $\XX$ has the $1$-strong Schur $p$-property.
\end{theorem}

\begin{proof}
Assume without loss of generality that $\It$ is countable and that $\XX$ is endowed with a uniformly continuous quasi-norm $\norm{\cdot}$. Let
\[
\varepsilon\colon(0,\infty)\to(0,\infty)
\]
be the modulus of continuity of $\norm{\cdot}$, $\rho$ be its diametral modulus, and $\gamma$ and $q$ be as in \eqref{eq:lqc}. Note that if $\XX_i$ is a $p$-Banach space for all $i\in \It$, then $\rho=2^{1/p}$, $\gamma = 1$, and $q=p$.

Let $\delta>0$ and $A$ be a $\delta$-separated normalized set. Pick $0<c<\delta/\rho$ and $C>\rho$. Since bounded sets of $\XX_i$, $i\in \It$, are relatively compact, by Cantor's diagonal method, there are $x_0\in\prod_{i\in \It} \XX_i$ and a sequence $(x_k)_{k=1}^\infty$ consisting of different vectors in $A$ such that $\lim_k x_k(i)=x_0(i)$ for all $i\in \It$. By Fatou's Lemma, $x_0 \in B_\XX$.

Pick $\varepsilon>0$. Let $\eta>0$ be such that $\varepsilon(\eta)\le \varepsilon$. There is $(\eta_k)_{k=0}^\infty$ in $(0,\infty)$ nonincreasing with $\lim_k\eta_k=0$ such that $\eta_0<\eta/\gamma$ and
\[
\enpar{\sum_{n=1}^\infty \abs{a_k}^q \eta_k^q}^{1/q} \le \enpar{\frac{\eta^q}{\gamma^q}-\eta_0^q}^{1/q}, \quad (a_k)_{k=1}^\infty\in c_{00}.
\]
Choose $y_0\in B_\XX$ finitely supported such that
\[
\norm{ x_0-y_0}\leq\eta_0.
\]

By the gliding-hump technique, passing to a subsequence we may assume that there is $(y_k)_{k=1}^\infty$ in $\XX$ such that
\[
\norm{ x_k-x_0-y_k} \le\eta_k,\quad k\in\Nat,
\]
and $(y_k)_{k=0}^\infty$ is disjointly supported. If $k$ and $n$ are distinct positive integers, then
\[
\abs{ \norm{ y_k-y_n} -\norm{ x_k-x_n} }\le \varepsilon(\rho(\eta_k+\eta_n)),
\]
whence $\delta \le \varepsilon(\rho(\eta_k+\eta_n))+\rho \max\enbrace{\norm{y_k} , \norm{y_n}}$.
Letting $n$ and $k$ tend to infinity, we infer that
\[
\rho \liminf_k \norm{ y_k}\ge \delta.
\]

Hence, we can assume that $\norm{ y_k}\ge c$ for all $k\in\Nat$. Similarly, since $\norm{y_k} \le \varepsilon(\eta_k) + \norm{x_k-x_0}$ for all $k\in\Nat$,
\[
\limsup_k \norm{y_k} \le \rho.
\]
Therefore, passing to another subsequence we can assume that $\norm{ y_k}\le C$ for all $k\in\Nat$.

Let $\alpha=(a_n)_{n=1}^\infty\in c_{00}$ be such that $\norm{\alpha}_p= 1$. If we set
\[
\overline\alpha\coloneqq\sum_{k=1}^\infty a_k, \quad x_\alpha\coloneqq\sum_{k=1}^\infty a_k\, x_k ,
\]
then $x_\alpha=y_\alpha+z_\alpha$, where
\begin{align*}
y_\alpha&= \overline\alpha y_0 +\sum_{k=1}^\infty a_k y_k, \\
z_\alpha&=\overline\alpha(x_0-y_0) +\sum_{k=1}^\infty a_k (x_k-y_k-x_0).
\end{align*}

Set $B=\enbrace{x_k \colon k\in\Nat}$. Since $\abs{\overline\alpha}\le 1$ by \eqref{eq:1lessp},
\[
\norm{z_\alpha}^q \le \gamma^q\enpar{\eta_0^q+\sum_{k=1}^\infty\abs{a_k}^q\eta_k^q}\le\eta^q.
\]

Consequently, $\abs{\norm{x_\alpha}-\norm{y_\alpha}}\le \varepsilon(\eta)\le\varepsilon$. Since
\[
\enpar{\sum_{k=1}^\infty \abs{a_k}^p \norm{ y_k}^p}^{1/p}
\leq\norm{ y_\alpha}
\leq\enpar{ \abs{\overline{\alpha}} \norm{ y_0}^p+\sum_{k=1}^\infty \abs{a_k}^p \norm{ y_k}^p}^{1/p},
%\leq\enpar{ \norm{ y_0}^p+\sum_{k=1}^\infty \abs{a_k}^p \norm{ y_k}^p}^{1/p},
\]
we have
\[
\lmz_p(B)\ge c-\varepsilon\quad \text{and} \quad (\enpar{1+ C^p}^{1/p}+\varepsilon) \umz_p(B)\ge 1.\]

Letting $\varepsilon$ tend to $0$ first, and then letting $c$ tend to $\delta/\rho$ and $C$ tend to $\rho$, we obtain $\slmz_p(A) \ge \delta/\rho$ and $\sumz_p(A) \ge \left(1 + \rho^p\right)^{-1/p}$.
\end{proof}

\begin{example}
Let $0<p\le 1$ and $0<q\le \infty$. Then the Besov space
\[
B_{q,p}\coloneqq\enpar{\bigoplus_{k=1}^\infty \ell_q^k}_{\ell_p}
\]
has the strong Schur $p$-property by \Cref{thm:lp}. Note that if $q<p$, then $B_{q,p}$ fails to be locally $p$-convex.
\end{example}

\begin{theorem}
Let $0<p\le 1$ and $(\XX_i)_{i\in I}$ be a finite family of quasi-Banach spaces. Set $\XX\coloneqq (\bigoplus_{i\in I} \XX_i)_{\ell_p}$.
\begin{enumerate}[label=(\roman*), widest=ii,leftmargin=*]
\item\label{it:finite:strong} If $\XX_i$ has the strong Schur $p$-property for all $i\in I$, then $\XX$ has the strong Schur $p$-property.
\item\label{it:finite:weak} If $\XX_i$ has the Schur $p$-property for all $i\in I$, then $\XX$ has the Schur $p$-property.
\end{enumerate}
\end{theorem}

\begin{proof}
Assume that $\XX$ is equipped with a continuous quasi-norm. Let $\rho$ be the diametral modulus of $\XX$, and let $\gamma$ and $q$ be as in \eqref{eq:lqc}.

To prove \ref{it:finite:strong}, we pick $1\le K<\infty$ such that $\XX_i$ has the $K$-strong Schur $p$-property for all $i\in\It$, $\delta>0$, and an infinite $\delta$-separated set $A\subset S_\XX$. Fix $\varepsilon\in (0,\delta)$ and a nonincreasing sequence $(\varepsilon_n)_{n=1}^\infty$ in $(0,\infty)$ such that $\lim_n \varepsilon_n=0$ and
\[
\enpar{\sum_{n=1}^\infty \abs{a_n}^q \varepsilon_n^q}^{1/q} \le \frac{\varepsilon}{\gamma} \enpar{\sum_{n=1}^\infty \abs{a_n}^p}^{1/p}, \quad (a_n)_{n=1}^\infty\in c_{00}.
\]

By \Cref{lem:dichothomy}, there are a sequence $(x_n)_{n=1}^\infty$ in $A$, $\Ft\subset \It$ and $y\in \XX_\Ft$ such that $(x_n(i))_{n=1}^\infty$ converges to $y(i)$ for all $i\in\Ft$, and $(x_n(i))_{n=1}^\infty$ is uniformly separated for all $i\in\Gt\coloneqq\It\setminus\Ft$. By infinite Ramsey's theorem, passing to a subsequence we find $(\delta_i)_{i\in\Gt}$ in $(0,\infty)$ and $(r_i)_{i\in\Gt}$ in $[0,\infty]$ such that
\begin{itemize}[leftmargin=*]
\item $\norm{(x_n-y)|_\Ft}\le \varepsilon_n$ for all $n\in\Nat$;
\item $\norm{x_n(i)}\in [r_i,\enpar{r_i^p +\varepsilon^p/\abs{\Gt}}^{1/p}]$ for all $i\in\Gt$ and $n\in\Nat$; and
\item $\norm{x_n(i)-x_m(i)}\in[\delta_i,\enpar{\delta_i^p +\varepsilon^p/\abs{\Gt}}^{1/p}]$ for all $i\in\Gt$ and all $n$, $m\in\Nat$ with $n\not=m$.
\end{itemize}
Set $\delta_0=\enpar{\sum_{i\in\Gt} \delta_i^p}^{1/p}$ and $r=\enpar{\sum_{i\in\Gt} r_i^p}^{1/p}$.
Pick $n$, $m\in\Nat$ with $n\not=m$. We have $1\ge \norm{x_n|_\Gt}\ge r$ and
\begin{align*}
\delta^p
&\le \norm{ (x_n-x_m)|_\Gt}^p + \norm{ (x_n-x_m)|_\Ft}^p\\
&\le \rho^p \max\enbrace{\norm{ (x_n-y)|_\Ft}^p, \norm{ (x_m-y)|_\Ft}^p} + \norm{ (x_n-x_m)|_\Gt}^p\\
&\le \rho^p \max\enbrace{\varepsilon_n^p, \varepsilon_m^p} +\delta_0^p + \varepsilon^p.
\end{align*}

Letting $m$ and $n$ tend to infinity, we obtain $\delta^p \leq \delta_0^p + \varepsilon^p$. Notice that this inequality implies $\Gt \neq \emptyset$.

By \Cref{ekvivStrongPSchur}, passing to another subsequence we infer that the set $B\coloneqq\enbrace{x_n \colon n\in\Nat}$ satisfies
\[
\lmz_p(B(i)) \ge \frac{ \enpar{\delta_i^p - \varepsilon^p/\abs{\Gt}}_+^{1/p}}{2^{1/p} K},
\quad
\umz_p(B(i)) \ge \frac{1}{K \enpar{r_i^p +\varepsilon^p/\abs{\Gt}}^{1/p}}
\]
for all $i\in\Gt$. Given $\alpha=(a_n)_{n=1}^\infty\in c_{00}$, using \eqref{eq:1lessp} we obtain
\begin{align*}
\norm{ \sum_{n=1}^\infty a_n x_n|_\Ft}
&\le \rho\max\enbrace{ \norm{ \sum_{n=1}^\infty a_n (x_n-y)|_\Ft} , \abs{\sum_{n=1}^\infty a_n} \norm{y}}\\
&\le \rho\max\enbrace{ \gamma\enpar{ \sum_{n=1}^\infty \abs{a_n}^q \varepsilon_n^q}^{1/q} ,\norm{\alpha}_p} \\
&\le \rho\max\enbrace{ \varepsilon ,1} \norm{\alpha}_p.
\end{align*}
Therefore,
\begin{align*}
\lmz_p(B) & \ge \frac{ \enpar{\delta_0^p - \varepsilon^p}_+^{1/p}}{2^{1/p} K} \ge \frac{ \enpar{\delta^p - 2\varepsilon^p}_+^{1/p}}{2^{1/p} K}, \\
\umz_p(B) & \ge \frac{1}{ K \enpar{r^p +\varepsilon^p + \rho^p \max\enbrace{ \varepsilon^p ,1}}^{1/p} }\ge
\frac{1}{ K \enpar{1+ \varepsilon^p + \rho^p \max\enbrace{ \varepsilon^p ,1}}^{1/p} }.
\end{align*}

Letting $\varepsilon$ tend to zero we obtain that
\[
\slmz_p(A) \ge \frac{\delta}{2^{1/p} K}, \quad \sumz_p(A) \ge \frac{1}{ K \enpar{1+\rho^p}^{1/p}}.
\]

The proof of \ref{it:finite:weak} is a replica of the proof of \ref{it:finite:strong} replacing in the appropriate places the estimates for $\lmz_p$, $\umz_p$, $\slmz_p$ and $\sumz_p$ by the fact that these quantities are positive.
\end{proof}

In order to study the Schur $p$-property of infinite sums we naturally have to restrict ourselves to $p$-Banach spaces. We start with a result inspired by \cite[Lemma 7.2]{KKS13}.

\begin{lemma}\label{ellPQuantitative}
Let $0<p\le 1$, $(\XX_i)_{i\in I}$ be a family of $p$-Banach spaces, and $A \subset \XX\coloneqq (\bigoplus_{i\in \It} \XX_i)_{\ell_p}$ be infinite. Then,
\begin{equation}\label{eq:lpQ}
I(A)\coloneqq \inf_{\Ft\in [\It]^{<\omega}} \sup_{x\in A} \norm{x|_{\It\setminus \Ft}}\leq \slmz_p(A).
\end{equation}

Moreover, if $A(i)\coloneqq\enbrace{x(i)\colon x\in A}$ is relatively compact for all $i\in I$, then
\[
\hsd(A) = \slmz_p(A) = I(A).
\]
\end{lemma}
\begin{proof}
To prove \eqref{eq:lpQ} we can assume that $I(A)>0$. Pick $0<\varepsilon< I(A)$. By definition, there exists $\Ft_0\in[\It]^{<\omega}$ such that $\norm{x|_{I\setminus \Ft_0}}< I(A)+\varepsilon$ for all $x\in A$, and for each $\Ft\in[I]^{<\omega}$ there is $y\in A$ such that $\norm{y|_{\It\setminus \Ft}}>I(A)-\varepsilon$.

We recursively find $(x_i)_{i=1}^\infty$ in $A$ and $(\Ft_i)_{i=1}^\infty$ in $[\It]^{<\omega}$ as follows. Assuming that $\Ft_{i-1}$ has been constructed, we pick $x_i$ such that
\[
\norm{ x_i|_{\It\setminus \Ft_{i-1}}} > I(A)-\varepsilon.\]
Then, we choose $\Ft_i \supset \Ft_{i-1}$ large enough so that
\[
\norm{x_i|_{\Ft_i\setminus \Ft_{i-1}}} >I(A)-\varepsilon.
\]
For each $i\in\Nat$ we have
\begin{multline*}
\norm{x_i|_{\It\setminus \Ft_0}-x_i|_{\Ft_i\setminus \Ft_{i-1}}}^p
=\norm{x_i|_{\It\setminus \Ft_0}}^p-\norm{x_i|_{\Ft_i\setminus \Ft_{i-1}}}^p\\
\le \enpar{I(A)+\varepsilon}^p - \enpar{I(A)-\varepsilon}^p\le (2\varepsilon)^p.
\end{multline*}

Pick $\lambda\coloneqq(\lambda_i)_{i=1}^\infty\in c_{00}$ and set $x\coloneqq\sum_{i=1}^\infty\lambda_i x_i$. We estimate
\[
\norm{x|_{\It\setminus \Ft_0} - \sum_{i=1}^\infty \lambda_i \, x_i|_{\Ft_i\setminus \Ft_{i-1}}}^p
\le \sum_{i=1}^\infty \lambda_i^p \norm{x_i|_{\It\setminus \Ft_0}-x_i|_{\Ft_i\setminus \Ft_{i-1}}}^p
\le (2\varepsilon)^p \norm{\lambda}_p^p.
\]
Therefore,
\begin{multline*}
\norm{x}^p
\ge \norm{x|_{\It\setminus \Ft_0}}^p\ge \norm{ \sum_{i=1}^\infty \lambda_i \, x_i|_{\Ft_i\setminus \Ft_{i-1}}}^p- (2\varepsilon)^p \norm{\lambda}_p^p\\
= \sum_{i=1}^\infty \lambda_i^p \norm{ x_i|_{\Ft_i\setminus \Ft_{i-1}}}^p - (2\varepsilon)^p \norm{\lambda}_p^p \ge \enpar{ (I(A)-\varepsilon)^p- (2\varepsilon)^p} \norm{\lambda}_p^p,
\end{multline*}
so that
\[
\slmz_p^p(A)\ge (I(A)-\varepsilon)^p- (2\varepsilon)^p.\] Letting $\varepsilon$ tend to zero, we are done.

By \Cref{fact:scAndChi}, to complete the proof, it suffices to prove that if $A(i)$ is relatively compact for all $i\in I$, then $\hsd(A)\leq I(A)$. Pick $c>c_0>I(A)$ and set $\varepsilon\coloneqq(c^p-c_0^p)^{1/p}$. There is $F\subset I$ finite such that
\[
\norm{x|_{\It\setminus \Ft}}\le c_0, \quad x\in A.
\]

We observe that $A|_\Ft$ is a relatively compact subset of $\XX_\Ft$. Therefore, there exists a finite $\varepsilon$-net for $A|_\Ft$, that is, there is $G\subset \XX_\Ft$ finite such that $\widehat{d}(A|_\Ft,G) < \varepsilon$. Given $x\in A$, there exists $g\in G$ with $\norm{x|_\Ft -g}< \varepsilon$. Hence,
\[
\norm{x-L_\Ft(g) }^p = \norm{x|_\Ft - g}^p + \norm{x_{\It\setminus \Ft}}^p < \varepsilon^p + c_0^p = c^p.
\]

This implies that $\widehat{d}(A,L_\Ft(G))\leq c$, whence $\hsd(A)\leq c$. Letting $c$ tend to $I(A)$ finishes the proof.
\end{proof}

\begin{theorem}\label{thm:schurPreservedByEllpsums}
Let $0<p\le 1$ and $(\XX_i)_{i\in \It}$ be a family of $p$-Banach spaces with the Schur $p$-property. Then $\XX\coloneqq (\oplus_{i\in \It} \XX_i)_{\ell_p}$ has the Schur $p$-property.
\end{theorem}
\begin{proof}
Let $A\subset \XX$ be a bounded and uniformly $\delta$-separated, $\delta>0$. We shall prove that $\slmz_p(A)>0$.

Let $I(A)$ be as is \Cref{ellPQuantitative}. If $I(A)>0$, we are done. Otherwise, there is $\emptyset\neq\Ft\subset \It$ finite such that $\norm{x|_{\It\setminus \Ft}}\le 3^{-1/p} \delta$ for all $x\in A$. Then, for all $x$, $y\in A$ with $x\not=y$,
\begin{multline*}
\norm{(x-y)|_\Ft}^p
= \norm{x-y}^p - \norm{(x-y)|_{\It\setminus \Ft}}^p \\
\ge \norm{x-y}^p - \norm{x|_{\It\setminus \Ft}}^p- \norm{y|_{\It\setminus \Ft}}^p
\ge \frac{1}{3}\delta^p,
\end{multline*}
whence $\max_{i\in \Ft} \norm{x(i)-y(i)}\ge \delta_0\coloneqq\abs{\Ft}^{-1/p} 3^{-1/p} \delta$. 

By infinite Ramsey's theorem, there exist $B\subset A$ infinite and $i\in \Ft$ such that $B(i)$ is $\delta_0$-separated. Therefore, $\slmz_p(B(i))>0$. Since $\slmz_p(A)\ge \slmz_p(B)\ge \slmz_p(B(i))$, we are done.
\end{proof}

The result concerning preservation of the strong Schur $p$-property under $\ell_p$-sums seems to be new even for the case of $p=1$. A similar result, in the setting of Banach spaces and involving a different quantitative version of the Schur property, was very recently and independently discovered in \cite[Theorem 7.4]{Kalenda25}.

\begin{theorem}
Let $(\XX_i)_{i\in \It}$ be a family of $p$-Banach spaces with the $K$-strong Schur $p$-property, where $p\in (0,1]$ and $K\geq 1$. Then $\XX\coloneqq (\oplus_{i\in \It} \XX_i)_{\ell_p}$ has the $K$-strong Schur $p$-property.
\end{theorem}

\begin{proof}
Let $\delta>0$ and $A\subset S_\XX$ be an infinite normalized $\delta$-separated set. Without loss of generality, we may assume that $A$ is countable and that $\It=\Nat$. By the small perturbation technique, we can also assume that every $x \in A$ is finitely supported. Let us enumerate the vectors of $A$ as $\enbrace{x_k \colon k\in\Nat}$.

Pick $0<\varepsilon<\delta$ and $(\varepsilon_i)_{i\in\It}$ in $(0,\infty)$ with $\sum_{i=1}^\infty \varepsilon_i^p =\varepsilon^p$. By infinite Ramsey's theory and Cantor's diagonal method, passing to a subsequence we may assume that there are $(\delta_i)_{i\in \It}$ and $(r_i)_{i\in \It}$ in $[0,\infty)$ such that
\begin{itemize}[leftmargin=*]
\item $ \delta_i \le \norm{x_k(i)-x_n(i) }\le \enpar{\delta_i^p+ \varepsilon_i^p}^{1/p}$ provided that $n>k\ge i$; and
\item $r_i\le \norm{x_k(i)}\le \enpar{r_i^p+ \varepsilon_i^p}^{1/p}$ provided that $k\ge i$.
\end{itemize}

Let us record three facts that gather the properties of $A$ that we will use.

\begin{claim}\label{claim:A}
$R\coloneqq\enpar{\sum_{i\in\It} r_i^p}^{1/p}\le 1$.
\end{claim}
Indeed, by Fatou's Lemma,
\[
R^p= \sum_{i=1}^\infty \lim_{k\to\infty} \norm{x_k(i)}^p \le \liminf_{k\to \infty} \sum_{i=1}^\infty \norm{x_k(i)}^p = 1.
\]

Set $\It'\coloneqq\{i\in\It \colon \delta_i>0\}$ and, given $\Ft\in[\It]^{<\omega}$,
\[
\delta_\Ft\coloneqq\enpar{\sum_{i\in\Ft} \delta_i^p}^{1/p}.
\]

\begin{claim}\label{claim:B}
For all $D\subset A$ infinite and all $\Ft\in[\It]^{<\omega}$ with $\Ft\cap\It'\not=\emptyset$ the map $x\mapsto x|_{\Ft}$, $x\in D$, is one-to-one, and there is $B\subset D$ infinite such that
\[
\lmz_p(B|_\Ft) \ge \frac{( \delta_\Ft^p -\varepsilon^p)_+^{1/p}}{2^{1/p} K}.
\]
\end{claim}
Indeed, let $(i_n)_{n=1}^N$ an enumeration of $\Ft'=\Ft\cap\It'$. Starting with $B_0=D$, we use \Cref{ekvivStrongPSchur} to recursively construct infinite sets $(B_{n})_{n=0}^N$ such that $B_{n}\subset B_{n-1}$ and
\[
\slmz_p(B_{n}(i_n)) \ge \frac{(\delta_{i_n}^p-\varepsilon_n^p)_+^{1/p}}{K 2^{1/p}}
\]
for all $i=1$, \dots, $N$. Set $B\coloneqq B_N$. We have
\[
\lmz_p^p(B|_\Ft) \ge \sum_{i\in\Ft'} \lmz_p^p(B(i)) \ge \frac{\delta^p_\Ft-\varepsilon^p}{2 K^p }.
\]

Note that \Cref{claim:B} implies that
\begin{equation}\label{eq:finitedelta}
\delta_0\coloneqq\enpar{\sum_{i\in\It} \delta_i^p}^{1/p}\le 2^{1/p} K<\infty.
\end{equation}
For convenience, we set $\eta\coloneqq\enpar{\delta^p-\delta_0^p-\varepsilon^p}_+^{1/p}$

\begin{claim}\label{claim:C}
For all $B\subset A$ infinite and all $\Ft\in[\It]^{<\omega}$, there is $x\in B$ such that
\[
\norm{x|_{\It\setminus\Ft}}\ge \frac{\eta}{2^{1/p}} , \quad\text{and}\quad \sup_{i\in\Ft} \norm{x(i)} \le \enpar{ r_i^p+\varepsilon_i^p}^{1/p}.
\]
\end{claim}

Indeed, the set $\It_\Ft\coloneqq\{k\in\Nat \colon k\ge \max(\Ft), \,x_k\in B\}$ is infinite. If $m$ and $n$ are distinct elements of $\It_\Ft$,
\begin{align*}
\norm{x_m|_{\It\setminus\Ft}}^p + \norm{x_n|_{\It\setminus\Ft}}^p
&\ge \norm{(x_m-x_n)|_{\It\setminus\Ft}}^p\\
&=\norm{x_m-x_n}^p-\sum_{i\in\Ft} \norm{(x_m(i)-x_n(i)}^p\\
& \ge \delta^p-\sum_{i\in\Ft} \enpar{\delta_i^p+\varepsilon_i^p}\\
&\ge \delta^p-\sum_{i\in\It} \enpar{\delta_i^p+\varepsilon_i^p} =\delta^p-\delta_0^p-\varepsilon^p.
\end{align*}
Hence, there is $k\in\{m,n\}$ such that $ \norm{x_k|_{\It\setminus\Ft}}\ge 2^{-1/p} \eta$.

Now, we are in a position to complete the proof. By \Cref{claim:A} and \eqref{eq:finitedelta}, let $\Ft_0\in[\It]^{<\omega}$ be such that
\[
\sum_{i\in\It\setminus \Ft_0} r_i^p \le \varepsilon^p, \quad \delta_{\Ft_0}^{p} \ge \delta_0^p-\varepsilon^p.
\]

By \Cref{claim:B}, there is $B\subset A$ infinite such that
\[
\lmz_p(B|_{\Ft_0}) \ge \frac{( \delta_0^p - 2 \varepsilon^p)_+^{1/p}}{2^{1/p} K}.
\]

Set $y_0=0$. We recursively construct $(y_k)_{k=1}^\infty$ in $B$ and $(\Ft_k)_{k=1}^\infty$ in $[\It]^{<\omega}$. Fix $k\in\Nat$ and assume that $y_{k-1}$ and $\Ft_{k-1}$ have been constructed.

We use \Cref{claim:C} to find $y_k\in B$ such that
\[
\norm{y_k|_{\It\setminus\Ft_{k-1}}}\ge \frac{\eta}{2^{1/p}}, \quad \sup_{i\in\Ft_{k-1}} \norm{ y_k(i)} \le \enpar{ r_i^p+\varepsilon_i^p}^{1/p}.
\]
Next choose $\Ft_k=\supp(y_k)\cup \Ft_{k-1}$.

Given $\alpha=(a_n)_{n=1}^\infty\in c_{00}$ we set
\[
y_\alpha\coloneqq\sum_{k=1}^\infty a_k \, y_k, \quad
u_\alpha\coloneqq y_\alpha|_{\Ft_0} + \sum_{k=1}^\infty a_k \, y_k|_{\Ft_k\setminus \Ft_{k-1}} , \quad v_\alpha \coloneqq\sum_{k=1}^\infty a_k \, y_k|_{\Ft_{k-1}\setminus \Ft_0},
\]
so that $y_\alpha=u_\alpha+v_\alpha$. On one hand,
\[
\norm{u_\alpha}^p = \norm{y_\alpha|_{\Ft_0}}^p + \sum_{k=1}^\infty \abs{a_k}^p \norm{y_k|_{\Ft_k\setminus \Ft_{k-1}}}^p
\ge \frac{\delta_0^p-2\varepsilon^p}{2 K^p} \norm{\alpha}_p^p + \frac{\eta^p}{2} \norm{\alpha}_p^p.
\]
On the other hand,
\[
\norm{v_\alpha}^p
\le\sum_{k=1}^\infty \abs{a_k}^p \sum_{i\in\Ft_{k-1}\setminus\Ft_0}\norm{y_k(i)}^p
\le \norm{\alpha}_p^p \sum_{i\in\It\setminus\Ft_0} r_i^p+\varepsilon_i^p
\le 2 \varepsilon^p \norm{\alpha}_p^p.
\]
Summing up, since $K\ge 1$,
\[
\norm{y_\alpha}^p \ge \norm{u_\alpha}^p- \norm{v_\alpha}^p \ge \enpar{ \frac{\delta^p-3\varepsilon^p}{2K^p}-2\varepsilon^p} \norm{\alpha}_p^p.
\]
We have proved that the set $E\coloneqq\{y_k \colon k\in \Nat\}$ satisfies
\[
\lmz_p(E) \ge c(\varepsilon,\delta,K,p)\coloneqq\enpar{ \frac{\delta^p-3\varepsilon^p}{2K^p}-2\varepsilon^p}_+^{1/p}.
\]
Therefore, $\slmz_p(A)\ge \lim_{\varepsilon\to 0^+} c(\varepsilon,\delta,K,p) =2^{-1/p} K^{-1} \delta$.
\end{proof}

We close this section by exhibiting that the strong Schur $p$-property is stronger than the Schur $p$-property.

\begin{example}
Given $0<\tau\le 1$, $0<p\le 1$, and $\alpha=(a_n)_{n=1}^\infty\in\FF^{\Nat}$, put
\[
\norm{\alpha}_{p,\tau}=\max\enbrace{ \tau \norm{\alpha}_p, \sup_{n\in\Nat} \abs{ \sum_{k=1}^n a_k}}.
\]
By \eqref{eq:1lessp}, $\tau\norm{\alpha}_p \le \norm{\alpha}_{p,\tau} \le \norm{\alpha}_p$. Hence, $\norm{\cdot}_{p,\tau}$ is a $p$-norm on $\ell_p$, and $\XX_{p,\tau}\coloneqq (\ell_p,\norm{\cdot}_{p,\tau})$ is just a renorming of $(\ell_p,\norm{\cdot}_p)$. Therefore, $\XX_{p,\tau}$ is a $p$-Banach space, and there is a constant $K\in(0,\infty)$ such that $\XX_{p,\tau}$ has the $K$-strong Schur $p$-property. We denote by $K_{p,\tau}$ the optimal such $K$.

The unit vector system $(\ee_n)_{n=1}^\infty$ regarded as a basis of $\XX_{p,\tau}$ is a spreading model, that is, it is isometrically equivalent to all its subsequences. Thus if we set
\[
E\coloneqq\{\ee_n \colon n\in\Nat\},
\]
then $\slmz_p(E)=\lmz_p(E)$. Put $\alpha_N=\sum_{j=1}^N (-1)^n \ee_n$ for all $N\in\Nat$. We have
\[
\norm{\alpha_N}_{p}=N^{1/p}, \quad \norm{\alpha_N}_{p,\tau}=\delta_{p,\tau,N}\coloneqq\max\{\tau N^{1/p},1\}, \quad N\in\Nat.
\]
In particular, $E$ is $\delta_{p,\tau,2}$-separated. Hence,
\begin{multline*}
K_{p,\tau} \ge \lmz_p^{{-1}}(E) \frac{\delta_{p,\tau,2}}{2^{1/p}}
\ge \sup_{N\in\Nat} \frac{ \norm{\alpha_N}_p}{\norm{\alpha_N}_{p,\tau}} \frac{\delta_{p,\tau,2}}{2^{1/p}}\\
=\sup_N \frac{\max\{\tau, 2^{-1/p}\} }{\max\{\tau , N^{-1/p}\} }
=\max\enbrace{1, \frac{1} {\tau 2^{1/p} }}.
\end{multline*}

Therefore, $\lim_{\tau\to 0^+} K_{p,\tau} =\infty$. Since $\XX\coloneqq\enpar{\bigoplus_{i=1}^\infty \XX_{p,1/i}}_{\ell_p}$ contains an isometric copy of $\XX_{p,1/i}$ for all $i\in\Nat$, $\XX$ fails to have the strong Schur $p$-property. In contrast, $\XX$ is a Schur $p$-space by \Cref{thm:schurPreservedByEllpsums}.
\end{example}
%-------------------------------
\section{Applications to Lipschitz free \texorpdfstring{$p$}{}-spaces for \texorpdfstring{$p\le 1$}{}}\label{free}\noindent
%------------------------------
Let $0<p\le 1$. Given a $p$-metric space $\Mt$ with a distinguished point $0$, there is a (real) $p$-Banach space $\F_p(\Mt)$ and and isometric embedding $\delta_{\Mt}\colon \Mt \to \F_p(\Mt) $ such that for every $p$-Banach space $\XX$ and every Lipschitz map $f\colon\Mt\rightarrow \XX$ with $f(0) = 0$ there exists a unique linear bounded map $T_{f}\colon\F_{p}(\Mt)\rightarrow \XX$ with $T_{f} \circ \delta_{\Mt} = f$.

Pictorially,
\[
\xymatrix{\Mt \ar[rr]^f \ar[dr]_{\delta_{\Mt}} & & \XX.\\
& \F_{p}(\Mt) \ar[ur]_{T_{f}} &}
\]

Moreover, $\norm{T_f} =\norm{f}_{\Lip}$. The space $\F_p(\Mt)$ does not essentially depend on the chosen distinguished point, so we will not mention it unless necessary. We refer the reader to \cite{AACD2018} for the basics of these relatives of the Lipschitz free spaces (also known in the literature as transportation cost spaces or Arens-Eells spaces), which were introduced in \cite{AlbiacKalton2009}. 

The authors of \cite{AACD2018} proved that $\F_p(\Mt)$ is finite-dimensional if and only if $\Mt$ is finite. Later on, they proved in \cite{AACD2020b} that $\ell_p$ embeds in $\F_p(\Mt)$ provided that $\Mt$ is infinite. So, if $\Mt$ is infinite and $q\not=p$, then $\F_p(\Mt)$ fails to have the Schur $q$-property. In this section, we propose the Schur $p$-property as a tool to study the geometry of Lipschitz free $p$-spaces over infinite $p$-metric spaces. We start with a decomposition theorem. First, we introduce some terminology.

Given a topological space $\Mt$ and an ordinal $\alpha$ we denote by $\Mt^{(\alpha)}$ its $\alpha$th Cantor--Bendixon derivative. The topological space $\Mt$ is said to be \emph{scattered} if there is an ordinal $\alpha$ such that $\Mt^{(\alpha)}=\emptyset$. It is known that if $\Mt$ is scattered, compact, and nonempty, then there is an ordinal $\beta<\omega_1$ such that $\Mt^{(\beta)}$ is finite and nonempty. We refer to \cite[Section VI.8]{DGZ93} for more details.

A metric space $\Mt$ is called \emph{proper} if every closed and bounded subspace is compact. The symbol $\XX \unlhd\YY$ means that the quasi-Banach $\XX$ is isomorphic to a complemented subspace of the quasi-Banach space $\YY$.

\begin{prop}[{cf.\@ \cite[Proposition 4.1]{AACD21}}]\label{prop:rank_decomposition_step}
Suppose that $\Mt$ is a complete metric space such that $\Mt^{(\alpha)}$ is nonempty and finite for some ordinal $\alpha$. Then there is a countable family $(\Mt_i)_{i=1}^\infty$ of closed bounded subsets of $\Mt$ such that $\Mt_i^{(\alpha)}$ is empty for each $i\in\Nat$ and
\[
\F_p (\Mt) \unlhd \enpar{ \bigoplus_{i=1}^\infty \F_p (\Mt_i)}_{\ell_p}
\]
for every $0<p\leq 1$.
\end{prop}

\begin{proof}
For $\alpha=1$ this was proved in \cite{AACD21}, the proof for general $\alpha$ is similar and therefore we omit it.
\end{proof}

The following general principle is inspired by a similar result from \cite{Dalet2015} concerning the approximation property.

\begin{theorem}\label{thm:preservedEllP}
Let $(P)$ be a property of $p$-Banach spaces, $0<p\le 1$, which is satisfied by any finite-dimensional space and is preserved by countable $\ell_p$-sums, isomorphisms, and complemented subspaces. If $\Mt$ is a scattered and proper metric space, then $\F_p(\Mt)$ has property $(P)$.
\end{theorem}

\begin{proof}
We initially prove the result when $\Mt$ is additionally compact. To that end, we proceed by transfinite induction on $\alpha<\omega_1$ such that $\Mt^{(\alpha)}$ is nonempty and finite.

For $\alpha=0$ the result is true, because then $\Mt$ is finite and therefore $\F_p(\Mt)$ is finite-dimensional.

Let $\alpha>0$ be such that the claim holds for any $\beta < \alpha$. If $\Nt$ is a closed subspace of $\Mt$ with $\Nt^{(\alpha)}=\emptyset$, then, since $\Nt$ inherits compactness from $\Mt$, there exists $\beta<\alpha$ such that $\Nt^{(\beta)}$ is finite and nonempty. Consequently, $\F_p(\Nt)$ has property $(P)$. By \Cref{prop:rank_decomposition_step}, $\F_p(\Mt)$ has property (P).

To prove the result in general, we note that by \cite[Theorem 3.5]{AACD21} we can still write
\[
\F_p(\Mt) \unlhd \enpar{ \bigoplus_{i=1}^\infty \F_p (\Mt_i)}_{\ell_p},
\]
where $(\Mt_i)_{i=1}^\infty$ are closed bounded subsets of $\Mt$. Since each metric space $\Mt_i$ is scattered and compact, $\F_p(\Mt_i)$ has property $(P)$. Hence, so does $\F_p(\Mt)$.
\end{proof}

The following result follows directly from \Cref{thm:preservedEllP}. Indeed, it is easy to see that the approximation property is preserved under $\ell_p$-sums and that it passes to complemented subspaces (see \cite[Propositions 2.9 and 2.10]{AACD21}). The same is true for the Schur $p$-property by \Cref{thm:schurPreservedByEllpsums}.

\begin{theorem}\label{cor:scatteredProper}
Suppose $\Mt$ is a scattered proper metric space. Then $\F_p(\Mt)$ has the approximation property and the Schur $p$-property for all $p\in (0,1]$.
\end{theorem}
%----------------------------------------
\section{Compact reduction}\label{cpctReduction}\noindent
%----------------------------------------
This section is geared towards proving the following theorem, which in particular answers \cite[Question 6.3]{AACD21} in the positive.

\begin{theorem}\label{thm:discreteCase}
Suppose $\Mt$ is a complete metric space with at most finitely many accumulation points. Then $\F_p(\Mt)$ has the approximation property for all $p\in(0,1]$.
\end{theorem}

Given $p\in(0,1]$, a \emph{metric} space $\Mt$, and $\Nt\subset \Mt$, we put
\[
\F_p(\Nt;\Mt)\coloneqq\enbrak{\delta_{\Mt}(x)\colon x\in \Nt} \subset \F_p(\Mt).
\]

The authors of \cite{CuthRaunig24} showed the existence of a constant $K(p)$ (in fact, $K(p)$ can be taken to be $7\cdot 12^{1/p-1}$), such that $\F_p(\Nt)$ is canonically $K(p)$-isomorphic to $\F_p(\Nt;\Mt)$, see \cite[Theorem 3.21]{CuthRaunig24}.

%thus answering a question that was raised in \cite{AACD2018}.

Since the family $(\delta_\Mt(x))_{x\in\Mt\setminus\{0\}}$ is linearly independent (see \cite{AACD2018}), for each $\gamma\in \Span (\delta_\Mt(\Mt))$ there is a minimal finite set $E\subset\Mt$ such that $0\in E$ and $\gamma\in\F_p(E;\Mt)$. We call $E$ the \emph{support of $\gamma$}, denoted by $\suppt(\gamma)$.

Given $E\subset\Mt$, we put
\[
[E]_\delta\coloneqq\enbrace{x\in \Mt\colon d(x,E)\leq \delta}.
\]

\begin{defn}
Let $p\in (0,1]$, $\Mt$ be a metric space and $\Nt\subset \Mt$. We say that $W\subset \F_p(\Nt;\Mt)$ has \emph{Kalton's property relative to $\Nt$} if for every $\varepsilon$, $\delta>0$, there exists a finite set $E\subset \Nt$ such that
\[
W\subset \F_p\enpar{[E]_\delta\cap \Nt;\Mt} + \varepsilon B_{\F_p(\Nt;\Mt)}.
\]

In the case $\Nt = \Mt$, we say that $W$ has \emph{Kalton's property}.
\end{defn}

Kalton’s property for subsets of Lipschitz free spaces $\F(\Mt)$ was established in \cite{ANPP}. The authors showed that every weakly precompact set in a Lipschitz free space possesses Kalton’s property (see \cite[Proposition 3.3]{ANPP}). Moreover, any set with Kalton’s property can be suitably approximated within a free space over a compact space (see \cite[Theorem 3.2]{ANPP}).

The combination of both results yields a ``compact reduction principle'' that enables us to deduce some properties of Lipschitz free spaces from their subspaces, regarded as Lipschitz free spaces over compact subsets.

Our proof of Theorem~\ref{thm:discreteCase} relies on a generalization of this compact reduction principle to Lipschitz free $p$-spaces $\F_p(\Mt)$ in the case when $\Mt$ is a discrete metric space. In this setting, knowing that compact sets have Kalton's property will suffice for our purposes.

\begin{lemma}\label{lem:cpct}
Let $p\in (0,1]$. Suppose that $\Mt$ is a metric space and that $K\subset \F_p(\Mt)$ is a relatively compact set. Then for every $\varepsilon>0$ there exists a finite set $E$ such that $K\subset \F_p(E;\Mt) + \varepsilon B_{\F_p(\Mt)}$. In particular, $K$ has Kalton's property.
\end{lemma}
\begin{proof}
Assume by contradiction that the conclusion does not hold, and let $\varepsilon>0$ witness the opposite. Put $E_0\coloneqq\{0\}$. We recursively find sequences $(\mu_n)_{n=1}^\infty$ from $K$, $(\gamma_n)_{n=1}^\infty$ from $ \Span (\delta_\Mt(\Mt))$, and finite sets $(E_n)_{n=0}^\infty$ such that for every $n\in\Nat$, we have
\begin{itemize}
\item $\mu_n\in K\setminus \enpar{\F_p(E_{n-1};\Mt} + \varepsilon B_{\F_p(\Mt)})$;
\item $\norm{\gamma_n - \mu_n} < 2^{-1/p} \varepsilon$; and
\item $E_n= E_{n-1}\cup \suppt(\gamma_n)$.
\end{itemize}
Then for $n<m$, we have
\[
\norm{\mu_n-\mu_m}^p\geq \norm{\gamma_n-\mu_m}^p - \frac{\varepsilon^p}{2}\geq \frac{\varepsilon^p}{2}.
\]

So, $(\mu_n)_{n=1}^\infty$ is a $(2^{-1/p} \varepsilon)$-separated separated sequence in $K$, which contradicts the relative compactness of $K$.
\end{proof}

We shall use a pair of of preliminary lemmas to establish our compact reduction principle for Lipschitz free $p$-spaces.

\begin{lemma}\label{lem:multOp}
Let $0<p\le 1$, $\Mt$ be a metric space, let $\Nt\subset \Mt$ be a nonempty and bounded subset, and $h\colon \Nt\to \Rea$  be a Lipschitz map. Pick an arbitrary point $0\in\Nt$ and set $R=\max_{x\in\Nt} d(x,0)$. Then there exists a linear operator
\[
S_h\colon \F_p(\Nt;\Mt)\to \F_p(\Mt)
\]
such that $S_h(\delta_\Mt(x)) = h(x)\delta_\Mt(x)$ for all $x\in \Nt$. Moreover, 
\[
\norm{S_h}\leq K(p) \enpar{ \norm{h}_\infty^p + \enpar{ R \norm{h}_{\Lip}}^p}^{1/p}.
\]
\end{lemma}

\begin{proof}
Set $L\coloneqq \enpar{ \norm{h}_\infty^p + \enpar{R \norm{h}_{\Lip}}^p}^{1/p}$. We use $0$ as the distinguished point of $\Nt$ and $\Mt$. Consider the map
\[
s\colon\Nt\to \F_p(\Mt), \quad
s(x)=h(x)\delta_\Mt(x).
\]

For $x$, $y\in \Nt$ we have
\begin{align*}
\norm{s(x)-s(y)} ^p & \leq \abs{h(x)}^p \norm{\delta_\Mt(x)-\delta_\Mt(y)}^p + \abs{h(x)-h(y)}^p\norm{\delta_\Mt(y)}^p\\
& \leq \enpar{ \norm{h}_\infty^p + \norm{h}_{\Lip}^p R^p} d^p(x,y).
\end{align*}

That is, $s$ is $L$-Lipschitz. By the universal property of Lipschitz free $p$-spaces, $s$ extends to a linear map $S\colon \F_p(\Nt)\to \F_p(\Mt)$ with $\norm{S} \leq L$. To complete the proof, it remains to invoke the definition of $K(p)$.
\end{proof}

\begin{lemma}\label{claim:subsetKalton}
Let $L\subset K\subset \Mt$, where $\Mt$ is a metric space.  Let $0<p\le 1$, $W\subset \F_p(K;\Mt)$ and $S\colon \F_p(K;\Mt)\to \F_p(L;\Mt)$ be a nonnull bounded linear operator. Suppose that $W$ has Kalton’s property relative to $K$, and that $S(\F_p(F;\Mt))\subset \F_p(F\cap L;\Mt)$ for every $F\subset K$. Then $S(W)$ also has Kalton's property relative to $L$.
\end{lemma}

\begin{proof}
Pick $\varepsilon>0$ and $\delta>0$. By assumption, there exists a finite set $E\subset K$ with
\[
W\subset \F_p([E]_{\delta/2};\Mt) + \frac{\varepsilon}{\norm{S}} B_{\F_p(K;\Mt)}.
\]

Let $E_0$ be the set of all $x\in E$ for which there exists $y(x)\in L$ with $d(x,y(x))\le \delta/2$. Set $F\coloneqq\{y(x)\colon x\in E_0\}$. 

Since, by the triangle law, $[E]_{\delta/2}\cap L\subset [F]_{\delta}$, we have
\begin{align*}
S(W)
&\subset S(\F_p([E]_{\delta/2};\Mt)) + \varepsilon B_{\F_p(L;\Mt)}\\
&\subset \F_p([E]_{\delta/2}\cap L;\Mt) + \varepsilon B_{\F_p(L;\Mt)}\\
&\subset \F_p([F]_{\delta};\Mt) + \varepsilon B_{\F_p(L;\Mt)}. \qedhere
\end{align*}
\end{proof}

We now complete the compact reduction principle.
Even though we are not able to remove the assumption of discreteness, quite a large portion of the proof does not require this assumption. Our result reads as follows.

\begin{theorem}\label{thm:KaltonsProperty} Let $\Mt$ be a complete metric space. Suppose $W\subset \F_p(\Mt)$ has Kalton's property. Then for every $\varepsilon>0$, there are
\begin{enumerate}[label=(\roman*),leftmargin=*,widest=iii]
\item\label{it:Bima:a} a sequence $(K_n)_{n=1}^\infty$ in $\Mt$ such that $\cap_{n=1}^\infty K_n$ is
compact;
\item\label{it:Bima:b} a (possibly unbounded) linear operator
\[
T\colon \Span(W) \to \bigcap_{n=1}^\infty \F_p(K_n;\Mt)
\]
with $\norm{\mu-T(\mu)}\leq \varepsilon$ for every $\mu\in W$; and
\item\label{it:Bima:c} bounded linear operators $T_n\colon\F_p(\Mt)\to \F_p(K_n;\Mt)$, $n\in\Nat$, such that $\lim_n T_n=T$ uniformly on $W$.
\end{enumerate}
Moreover, if $\Mt$ is discrete, then the sequence $(K_n)_{n=1}^\infty$ may be chosen to be constant. That is, there exists a finite set $K\subset \Mt$ such that $K_n=K$ for every $n\in\Nat$.
\end{theorem}

\begin{proof}
First of all, we may assume without loss of generality that $\Mt$ is a complete, bounded metric space. This follows essentially from the ``bounded reduction'' principle developed in \cite{AACD2022}. Let us give more details. There exists a bounded metric space $\Nt$, together with a surjective homeomorphism $B\colon\Mt\to\Nt$ and a surjective linear isomorphism $I\colon \F_p(\Mt)\to \F_p(\Nt)$, which satisfy the following additional properties:
\begin{itemize}
\item $B$ is Lipschitz.
\item For any $K\subset \Mt$ we have $I(\F_p(K;\Mt))\subset \F_p(B(K);\Nt)$.
\item for any $L\subset \Nt$ we have $I^{-1}(\F_p(L;\Nt))\subset \F_p(B^{-1}(L);\Mt)$.
\end{itemize}

We infer that $I(W)\subset \F_p(\Nt)$ inherits Kalton's property from $W$. Hence, assuming that Theorem~\ref{thm:KaltonsProperty} holds for bounded metric spaces, for any $\varepsilon>0$ we obtain $(K_n)_{n=1}^\infty$, $(T_n)_{n=1}^\infty$, and $T$ satisfying \ref{it:Bima:a}, \ref{it:Bima:b} and \ref{it:Bima:c} relative to $\Nt$, $I(W)$ and $\varepsilon/\norm{I^{-1}}$. It is routine to check that $B^{-1}(K_n)$, $I^{-1}\circ T_n\circ I$, $n\in\Nat$ and $I^{-1}\circ T\circ I$ witness that Theorem~\ref{thm:KaltonsProperty} holds for the space $\Mt$, including the ``moreover'' part.

Fix $\varepsilon>0$. Put $\delta_0\coloneqq\diam(\Mt)$ and, for every $n\in\Nat$,
\[
\varepsilon_n\coloneqq 2^{-n/p} \varepsilon, \quad \delta_n\coloneqq\delta_0\enpar{\enpar {K(p)\varepsilon_n}^{-p}-2}^{-1/p}.
\]

Further, let $W_0\coloneqq W$, $K_0\coloneqq\Mt$, and $S_0\coloneqq T_0$ be the identity operator on $\F_p(\Mt)$. We shall recursively construct for each $n\in\Nat$ a finite set $E_n\subset\Mt$, a closed set $K_n\subset \Mt$, $W_n\subset \F_p(K_n;\Mt)$, and linear operators $S_n\colon \F_p(K_{n-1};\Mt)\to \F_p(K_n;\Mt) $ and $T_n\colon \F_p(\Mt)\to \F_p(K_n;\Mt) $ such that
\begin{enumerate}[label=(\arabic*),widest=a]
\item\label{it:Bima:u} $K_n=[E_n]_{2\delta_n}\cap K_{n-1}$;
\item\label{it:Bima:s} $S_n(\F_p(F;\Mt))\subset \F_p(F\cap K_n;\Mt)$ for every $F\subset K_{n-1}$;
\item\label{it:Bima:t} $S_n$ is the identity map on $\F_p([E_n]_{\delta_n}\cap K_{n-1};\Mt)$;
\item\label{it:Bima:p} $\norm{\mu-S_n(\mu)}\leq \varepsilon_n$ for every $\mu\in W_{n-1}$;
\item\label{it:Bima:r} $T_n=S_n \circ T_{n-1}$; and
\item\label{it:Bima:q} $W_n=T_{n}(W)$.
\end{enumerate}

Assume that for fixed $n\in \Nat$, the $(n-1)$th step has been completed.
From \Cref{claim:subsetKalton}, we deduce that $W_{n-1}$ has Kalton's property relative to $K_{n-1}$, and therefore we find a finite set $E_n\subset K_{n-1}$ with
\[
W_{n-1}\subset \F_p([E_n]_{\delta_n}\cap K_{n-1};\Mt) + \varepsilon_n^2B_{\F_p(K_{n-1};\Mt)}.
\]

Define $K_n$ by \ref{it:Bima:u}. We use the Whitney--McShane extension theorem to pick a $({1}/{\delta_n})$-Lipchitz function $h_n\colon K_{n-1}\to [0,1]$ satisfying
\begin{enumerate}[label=(\alph*),widest=a]
\item\label{it:Bima:e} $h_n(x)=1$ for all $x\in [E_n]_{\delta_n}\cap K_{n-1}$; and
\item\label{it:Bima:d} $h_n=0$ on $K_{n-1}\setminus K_n$.
\end{enumerate}

Now, pick the bounded linear operator $S_n\coloneqq S_{h_n}$ from Lemma~\ref{lem:multOp}, which satisfies \ref{it:Bima:s} by \ref{it:Bima:d} and also \ref{it:Bima:t} by \ref{it:Bima:e}. In particular,
\[
S_n\enpar{\F_p(K_{n-1};\Mt)} \subset \F_p(K_n;\Mt).
\]

To prove \ref{it:Bima:p}, let $\mu \in W_{n-1}$ be given. We pick $\lambda\in \F_p([E_n]_{\delta_n}\cap K_{n-1};\Mt)$ with $\norm{\mu-\lambda}<\varepsilon_n^2$. It follows that
\begin{align*}
\norm{\mu-S_n(\mu)}^p & \leq \varepsilon_n^{2p}\enpar{1 + \norm{S_n}^p} + \norm{\lambda - S_n(\lambda)} = \varepsilon_n^{2p}\enpar{1 + \norm{S_n}^p}\\
& \leq \enpar{\varepsilon_n^{2}K(p)}^p\enpar{2+\enpar{\frac{\delta_0}{\delta_n}}^p}\leq \varepsilon_n^p.
\end{align*}

Finally, we define $T_n$ by \ref{it:Bima:r}, and $W_n$ by \ref{it:Bima:q}.

Let $\VV$ be the vector space consisting of all $\mu\in\F_p(\Mt)$ such that $(T_n(\mu))_{n=1}^\infty$ converges. We set
\[
T\colon \VV \to \bigcap_{n=1}^\infty \F_p(K_n;\Mt), \quad T(\mu)=\lim_{n\to\infty} T_n(\mu).
\]

By construction, $\krt(K_n)\le 2 \delta_n$ for all $n\in\Nat$, and $\lim_n \delta_n=0$. Therefore, the set
\begin{equation}\label{eq:defK}
K\coloneqq\bigcap_{n\in\Nat} K_n
\end{equation}
is compact.

For $\mu\in W$, $n\in\ZZ$ nonnegative, and $k\in\Nat$, we have
\begin{align*}
\norm{T_{n+k} (\mu)- T_n(\mu)}^p
&\le \sum_{j=n+1}^{n+k} \norm{ T_{j}(\mu)-T_{j-1}(\mu) }^p\\
&= \sum_{j=n+1}^{n+k} \norm{ S_j\enpar{T_{j-1}(\mu)}-T_{j-1}(\mu) }^p\\
&\le \sum_{j=n+1}^{n+k} \varepsilon_j^{p} \le 2^{-n} \varepsilon^{p}.
\end{align*}

This implies that $(T_n(\mu))_{n=1}^\infty$ is a Cauchy sequence for every $\mu\in W$. Hence, $W \subset \VV$, and for all $\mu \in W$ and all nonnegative integers $n$, we have $\norm{T(\mu) - T_n(\mu)} \leq 2^{-n} \varepsilon$. This inequality yields $\lim_n T_n=T$ uniformly on $W$ and, choosing $n=0$, $\norm{T (\mu)- \mu} \le \varepsilon$ for all $\mu \in W$.

To conclude the proof, we claim that if $\Mt$ is discrete then $K = K_n$ for some $n\in\Nat$.

Indeed, if this were not the case, then $F_n\coloneqq K_n\setminus K\neq\emptyset$ for all $n\in\Nat$. Since $\lim_n \krt(F_n)=0$ and, by assumption, $F_n$ is closed for all $n\in\Nat$, Kuratowski’s intersection theorem implies that $\bigcap_{n=1}^\infty F_n\neq \emptyset$. This, however, contradicts \eqref{eq:defK}.

Since $K$ is compact, it is finite. Passing to a subsequence, we may assume that $K_n=K$ for all $n\in\Nat$.
\end{proof}

We point out that, as it happens in the Banach space setting, a quasi-Banach space has the approximation property if and only the identity can be uniformly approximated on compact sets by finite-rank operators.

\begin{proof}[Proof of Theorem~\ref{thm:discreteCase}]
By \cite[Proposition 4.1]{AACD21} (see e.g.\@ \cite[comment before Question 6.3]{AACD21}), it suffices to consider the case when $\Mt$ is a complete discrete metric space.

Let $W\subset \F_p(\Mt)$ be compact and take any $\varepsilon>0$. Combining Lemma~\ref{lem:cpct} with Theorem~\ref{thm:KaltonsProperty} yields a finite set $K$, a function $T\colon W\to \F_p(K;\Mt)$ with $\norm{T-\Id_W}_\infty\le \varepsilon/2$, and a bounded linear operator $S\colon \F_p(\Mt)\to \F_p(K;\Mt)$ with $\norm{S|_W-T}_\infty\le \varepsilon/2$. Consequently, $\norm{S|_W- \Id_W}_\infty\le \varepsilon$.

Since $S$ is a finite-rank operator, we are done.
\end{proof}
% ------------------------------------------------------------------------
\section{Open questions}\noindent
% ------------------------------------------------------------------------
In view of \Cref{cor:scatteredProper} we wonder for the behaviour of Lipschitz free $p$-spaces $\F_{p}(\Mt)$ over compact non-scattered metric spaces $\Mt$. In particular, we are interested in the case where $\Mt = [0,1]$.

\begin{question}\label{qt:SchurFp01}
Let $0<p<1$. Does $\F_p([0,1])$ have the Schur $p$-property?
\end{question}

We emphasize the connection of \Cref{qt:SchurFp01} with the problem raised in \cite{AlbiacKalton2009} whether there exist $0<p<1$ and a metric space $\Mt$ such that $\F_p(\Mt)$ contains a complemented subspace isomorphic to an infinite-dimensional Banach space.

In fact, no infinite-dimensional Banach space is known to linearly embed into a Lipschitz free $p$-space for $p<1$. Taking into account that every separable metric space embeds into $c_0$, we wonder if there is any Banach space that linearly embeds into $\F_p(c_0)$.

Note that the answer to \Cref{qt:SchurFp01} is negative for $p=1$. In fact, $\F([0,1])$ is isomorphic to $L_1$, and $L_1$ contains a copy of $\ell_2$. Note also that, given $0<p<1$, $\F_p([0,1],\abs{\cdot}^{1/p})$ is isomorphic to $L_p$ (see \cite{AACD2018}), and $L_p$ is not a $p$-Schur space since it contains $\ell_2$ as well. However, the anti-snowflacking $([0,1],\abs{\cdot}^{1/p})$ is a $p$-metric space that fails to be a metric space. In fact, its metric envelope is null (see \cite{AACD2018}).

The \emph{Banach envelope} of a quasi-Banach space $\XX$ is Banach space $\widehat{\XX}$ together with a linear contraction $J_\XX\colon \XX\to \widehat{\XX}$, called the \emph{envelope map}, satisfying the universal property that for any Banach space $\YY$ and any linear contraction $T\colon \XX\to \YY$, there is a linear contraction $\widehat{T}\colon \widehat{\XX} \to \YY$ such that $\widehat{T}\circ J_\XX=T$. Since the Banach envelope of $\F_p([0,1])$ is $\F_1([0,1])$, \Cref{qt:SchurFp01} connects with the inheritability of Schur properties by envelopes.

\begin{question}\label{qt:SchurEnve}
Let $0<p<1$ and $\XX$ be a quasi-Banach space with the Schur $p$-property. Does $\widehat{\XX}$ have the Schur property?
\end{question}

An important point to make here is that envelope maps may not be one-to-one. In fact, there are quasi-Banach spaces such as $L_p$ for $p<1$ with trivial dual and, hence, trivial Banach envelope. Hence, as a particular case of \Cref{qt:SchurEnve}, we ask about the existence of $p$-Schur spaces whose dual space is null.

It will also be interesting to determine whether the Schur $p$-property and the strong Schur $p$-property are different in the setting of Lipschitz free $p$-spaces.

\begin{question}\label{qt:SchurvsSCh}
Let $0<p\le 1$. Is there a metric space $\Mt$ such that $\F_p(\Mt)$ has the Schur $p$-property but fails to have the strong Schur $p$-property?
\end{question}

Any Lipschitz free $p$-space over a metric space is isomorphic to a Lipschitz free $p$-space over a bounded metric space \cite{AACD2022}, and the techniques from \cite{AACD21} allow to construct a metric space $\Mt$ from a countable family $(\Mt_i)_{i\in \It}$ of bounded metric spaces such that $\F_p(\Mt)$ is isomorphic to $\enpar{\bigoplus_{i\in\It} \F_p(\Mt_i) }_{\ell_p}$.

Hence, \Cref{qt:SchurvsSCh} is equivalent to asking about the existence for each $K\in[1,\infty)$ of a bounded metric space $\Mt$ such that $\F_p(\Mt)$ fails to have the $K$-strong Schur $p$-property. In particular, we wonder if for each $K\in[1,\infty)$ there is a uniformly separated bounded metric space $\Mt$ such $\F_p(\Mt)$ fails to have the $K$-strong Schur $p$-property despite being isomorphic to $\ell_p$ (see \cite[Theorem 4.14]{AACD2018}).

We note that very recently Kalenda (\cite{Kalenda25}) found an example of a uniformly separated bounded metric space $\Mt$ such that $\F(\Mt)$ does not have $1$-strong Schur property, see \cite[Example 8.5]{Kalenda25}. However, the techniques to prove this are using duality techniques, so we do not know whether an analogue for $p<1$ may be proved as well.

In connection with the compact reduction principle discussed in Section~\ref{cpctReduction}, we pose several further related questions. For each of the questions above, a positive answer is known when $p=1$. However, the existing proofs rely on duality techniques. Thus, even in the case $p=1$, a positive answer would require a new proof.

\begin{question}\label{qt:intersection}
Let $p\in(0,1)$, $\Mt$ be a complete metric space, and $(K_n)_{n=1}^\infty$ be a nonincreasing sequence of closed subsets of $\Mt$. Does the following inclusion hold?
\[
\bigcap_{n=1}^\infty \F_p(K_n;\Mt)\subset \F_p\enpar{\bigcap_{n=1}^\infty K_n;\Mt}
\]
\end{question}

\begin{question}\label{qt:firstPartCpctReduction} Let $p\in (0,1)$ and $\Mt$ be a bounded metric space. Suppose that a bounded set $W\subset \F_p(\Mt)$ does not have Kalton's property.

Does $W$ contain a sequence equivalent to the canonical basis of $\ell_p$?
\end{question}

\begin{question}\label{qt:discreteSchur}Let $p\in(0,1)$ and $\Mt$ be a bounded complete discrete metric space. Does $\F_p(\Mt)$ have the Schur $p$-property?
\end{question}

Question~\ref{qt:intersection} asks for an analogue of the Intersection theorem, see \cite[Theorem 2.1]{APPP20}. A positive answer would lead to a strengthening of Theorem~\ref{thm:KaltonsProperty} in the form of the corresponding \cite[Theorem 3.2]{ANPP}.

Since weakly compact sets could not give valuable information on the geometry of non-locally convex spaces, a positive answer to Question~\ref{qt:firstPartCpctReduction} would be a nice substitute for \cite[Proposition 3.3]{ANPP}. In fact, it implies a positive answer to Question~\ref{qt:discreteSchur} along the lines of the proof of \cite[Corollary 2.7]{ANPP}.
% ------------------------------------------------------------------------
%\emergencystretch=0.5em
\bibliography{schur}
\bibliographystyle{plain}
% ------------------------------------------------------------------------
\end{document}